\documentclass{amsart} 
\usepackage{amssymb,amsmath,latexsym,times,xcolor,hyperref,tikz,enumitem}

\hypersetup{colorlinks=true, linkcolor=blue, citecolor=blue, urlcolor=blue}

\numberwithin{equation}{section}

\theoremstyle{plain}
\newtheorem{thm}{Theorem}[section]

\newtheorem{cor}[thm]{Corollary}

\newtheorem{lemma}[thm]{Lemma}

\newtheorem*{lem*}{Lemma}
\newtheorem*{thm*}{Theorem}
\newtheorem*{cor*}{Corollary}

\theoremstyle{definition}

\newcommand{\del}{\backslash}
\newcommand{\meet}{\land}
\newcommand{\join}{\lor}
\DeclareMathOperator{\cl}{cl}

\title[A Configuration-Preserving Matroid Construction]{A Construction
  that Preserves the Configuration of a Matroid, with Applications to
  Lattice Path Matroids} \author[J.~Bonin]{Joseph E.~Bonin} \address
{Department of Mathematics\\ The George Washington University\\
  Washington, D.C.\ 20052, USA} \email {jbonin@gwu.edu} \date{\today}
\author[A.~de Mier]{Anna de Mier} \address{Departament de Matem\`atiques and IMTECH\\ Universitat Polit\`ecnica de Catalunya\\ Jordi Girona
  1--3, 08034\\ Barcelona, Spain} \email{anna.de.mier@upc.edu}


\begin{document}

\begin{abstract}
  The configuration of a matroid $M$ is the abstract lattice of cyclic
  flats (flats that are unions of circuits) where we record the size
  and rank of each cyclic flat, but not the set.  One can compute the
  Tutte polynomial of $M$, and stronger invariants (notably, the
  $\mathcal{G}$-invariant), from the configuration.  Given a matroid
  $M$ in which certain pairs of cyclic flats are non-modular, we show
  how to produce a matroid that is not isomorphic to $M$ but has the
  same configuration as $M$.  We show that this construction applies
  to a lattice path matroid if and only if it is not a fundamental
  transversal matroid, and we enumerate the connected lattice path
  matroids on $[n]$ that are fundamental; these results imply that,
  asymptotically, almost no lattice path matroids are Tutte unique.
  We give a sufficient condition for a matroid to be determined, up to
  isomorphism, by its configuration.  We treat constructions that
  yield matroids with different configurations where each matroid is
  determined by its configuration and all have the same
  $\mathcal{G}$-invariant.  We also show that for any lattice $L$
  other than a chain, there are non-isomorphic transversal matroids
  that have the same configuration and where the lattices of cyclic
  flats are isomorphic to $L$.
\end{abstract}

\maketitle

\section{Introduction}

The configuration of a matroid is the abstract lattice that is formed
by the cyclic flats (the flats that are unions of circuits) together
with the size and rank of each cyclic flat, without recording the sets
that are cyclic flats.  The configuration is important in part because
it contains all of the data that is needed to compute many enumerative
matroid invariants.  The most well-known of these invariants is the
Tutte polynomial (see, e.g., \cite{Tutte,handbook}).  For a matroid
$M$ on the set $E(M)$, its \emph{Tutte polynomial} is defined to be
$$T(M;x,y) = \sum_{A\subseteq
  E(M)}(x-1)^{r(M)-r(A)}(y-1)^{|A|-r(A)},$$ which is a generating
function for the pairs $(|A|,r(A))$ as $A$ ranges over all subsets of
$E(M)$.  Derksen \cite{G-inv} introduced a strictly stronger
enumerative invariant, the $\mathcal{G}$-invariant, denoted
$\mathcal{G}(M)$, that records, for each permutation
$\pi=(e_1,e_2,\ldots,e_n)$ of $E(M)$, the $0,1$-vector
$(r_1,r_2,\ldots,r_n)$ of rank increases when the elements of $E(M)$
are added in the order that $\pi$ gives, that is,
$r_i=r(\{e_1,e_2,\ldots,e_i\})-r(\{e_1,e_2,\ldots,e_{i-1}\})$.  Bonin
and Kung \cite{catdata} showed that $\mathcal{G}(M)$ is equivalent to
recording, for each $(r(M)+1)$-tuple of integers
$(d_0,d_1,\ldots,d_{r(M)})$, the number of flags
$\cl_M(\emptyset)=F_0\subsetneq F_1\subsetneq \cdots\subsetneq
F_{r(M)}=E(M)$ of flats of $M$ for which $d_0=|F_0|$ and
$d_i=|F_i-F_{i-1}|$ for $1\leq i\leq r(M)$.  Eberhardt \cite{config}
proved that $T(M;x,y)$ can be computed from the configuration of $M$,
and Bonin and Kung \cite{catdata} showed that the same holds for
$\mathcal{G}(M)$.

This paper develops a new line of inquiry: given a matroid $M$, under
what conditions, and how, can one construct a matroid that is not
isomorphic to $M$ and yet has the same configuration as $M$?
Complementary to that, we also consider matroids $M$ for which all
matroids that have the same configuration as $M$ are isomorphic to
$M$; such matroids are \emph{configuration unique}.  Two related
notions also play roles: a matroid $M$ is \emph{Tutte unique} if
$T(M;x,y)=T(N;x,y)$ implies that $N$ is isomorphic to $M$, and $M$ is
\emph{$\mathcal{G}$ unique} if $\mathcal{G}(M)=\mathcal{G}(N)$ implies
that $N$ is isomorphic to $M$.  Since the Tutte polynomial can be
computed from the $\mathcal{G}$-invariant, which can be computed from
the configuration, Tutte-unique matroids are $\mathcal{G}$ unique, and
$\mathcal{G}$-unique matroids are configuration unique.  See
\cite{TUnTEq} for a survey of Tutte uniqueness.

In Theorem \ref{thm:twofilters}, for a matroid $M$ in which certain
pairs of cyclic flats are non-modular, we show how to construct a
matroid that has the same configuration but is not isomorphic to $M$.
In Section \ref{sec:lpm}, we prove that that construction applies to a
lattice path matroid if and only if it is not a fundamental (or
principal) transversal matroid.  We characterize lattice path matroids
that are fundamental transversal matroids in several ways, and we show
that such matroids are configuration unique.  In Section
\ref{sec:counting}, we show that, for $n\geq 2$, there are
$(3^{n-2}+1)/2$ connected lattice path matroids on
$[n]=\{1,2,\ldots,n\}$ that are fundamental transversal matroids, and
we refine that count by rank; several well-known combinatorial
sequences, such as Pell and Delannoy numbers, arise naturally in this
work.  It follows that, asymptotically, almost no lattice path
matroids are configuration unique.  Theorem \ref{thm:congifrecon}
gives a sufficient condition for a matroid to be configuration unique.
Section \ref{sec:confrecon} also
% gives examples
treats several constructions
of matroids that are
configuration unique but not $\mathcal{G}$ unique.  While nested
matroids (matroids for which the lattice of cyclic flats is a chain)
are known to be Tutte unique, in Theorem \ref{thm:latticeresult}, we
show that for every lattice $L$ that is not a chain, there are
transversal matroids whose lattices of cyclic flats are isomorphic to
$L$ and to which the construction in Theorem \ref{thm:twofilters}
applies.

\section{Background}\label{sec:background}

For general matroid theory background, see Oxley \cite{oxley}. We
adopt the matroid notation used there.  All matroids and lattices
considered in this paper are finite.  We use $[a,b]$ to denote the
interval $\{a,a+1,\ldots,b\}$ in the set $\mathbb{Z}$ of integers, and
we simplify $[1,n]$ to $[n]$.

\subsection{Cyclic flats, configurations,  modular pairs, and
  principal extensions}

A \emph{cyclic set} of a matroid $M$ is a (possibly empty) union of
circuits; equivalently, $X\subseteq E(M)$ is cyclic if $M|X$ has no
coloops.  A \emph{cyclic flat} is a flat that is cyclic.  The set of
cyclic flats of $M$ is denoted $\mathcal{Z}(M)$.  With inclusion as
the order, $\mathcal{Z}(M)$ is a lattice: for $A,B\in \mathcal{Z}(M)$,
the join $A\join B$ is $\cl(A\cup B)$ and the meet $A\meet B$ is
$(A\cap B)-C$ where $C$ is the set of coloops of $M|(A\cap B)$.
Routine arguments show that, for all $X\subseteq E(M)$,
\begin{equation}\label{eq:cftorank}
  r(X)=\min\{r(F)+|X-F|\,:\,F\in\mathcal{Z}(M)\},
\end{equation}
so $M$ is determined by its ground set $E(M)$ along with the cyclic
flats of $M$ and their ranks, that is, by the set
$\{E(M)\}\cup\{(F,r(F))\,:\,F\in \mathcal{Z}(M)\}$.  One cyclic flat
$F$ that yields the minimum in Equation (\ref{eq:cftorank}) is the
closure of the union of the circuits of $M|X$.  We will use the
following result from \cite{sims, cycflats}, which characterizes
matroids from the perspective of cyclic flats and their ranks.

\begin{thm}\label{thm:axioms}
  For a collection $\mathcal{Z}$ of subsets of a set $E$ and a
  function $r:\mathcal{Z}\to \mathbb{Z}$, there is a matroid $M$ on
  $E$ with $\mathcal{Z}(M)=\mathcal{Z}$ and $r_M(X) =r(X)$ for all
  $X\in\mathcal{Z}$ if and only if
  \begin{itemize}
  \item[\emph{(Z0)}] $(\mathcal{Z},\subseteq)$ is a lattice,
  \item[\emph{(Z1)}] $r(0_{\mathcal{Z}})=0$, where $0_{\mathcal{Z}}$
    is the least set in $\mathcal{Z}$,
  \item[\emph{(Z2)}] $0<r(Y)-r(X)<|Y-X|$ for all sets $X,Y$ in
    $\mathcal{Z}$ with $X\subsetneq Y$, and
  \item[\emph{(Z3)}] for all pairs of sets $X,Y$ in $\mathcal{Z}$ (or,
    equivalently, just incomparable sets in $\mathcal{Z}$),
    $$r(X\join Y) + r(X\meet Y) + |(X\cap Y) - (X\meet Y)|\leq
    r(X)+r(Y).$$
  \end{itemize}
\end{thm}

By Equation (\ref{eq:cftorank}), for matroids $M$ and $N$, a function
$\phi:E(M)\to E(N)$ is an isomorphism of $M$ onto $N$ if and only if
$\phi$ is a bijection, $\phi$ maps $\mathcal{Z}(M)$ onto
$\mathcal{Z}(N)$, and $r_M(A)=r_N(\phi(A))$ for all
$A\in\mathcal{Z}(M)$.

A set $A$ is cyclic in a matroid $M$ if and only if $E(M)-A$ is a flat
of the dual matroid $M^*$ since $A$ being a union of circuits of $M$
is equivalent to $E(M)-A$ being an intersection of hyperplanes of
$M^*$.  Thus, $X$ is a cyclic flat of $M$ if and only if $E(M)-X$ is a
cyclic flat of $M^*$, and so $\mathcal{Z}(M^*)$, the lattice of cyclic
flats of $M^*$, is isomorphic to the order dual of $\mathcal{Z}(M)$.

The \emph{configuration} of a matroid $M$ is a $4$-tuple
$(L, s, \rho,|E(M)|)$, where $L$ is a lattice and $s:L\to\mathbb{Z}$
and $\rho:L\to\mathbb{Z}$ are functions such that there is an
isomorphism $\phi: L\to \mathcal{Z}(M)$ for which $s(x)=|\phi(x)|$ and
$\rho(x)=r(\phi(x))$ for all $x\in L$.  Many $4$-tuples can satisfy
these properties, but they all contain the same data, so we view them
as the same.  Two matroids \emph{have the same configuration} if some
$4$-tuple $(L, s, \rho,n)$ is the configuration of both.  As Figure
\ref{fig:sameconfig} illustrates, non-isomorphic matroids can have the
same configuration.

\begin{figure}
  \centering
  \begin{tikzpicture}[scale=1]
    \filldraw (0,1) node[above=2] {\small$1$} circle (2.5pt); %
    \filldraw (1,1) node[above=2] {\small$2$} circle (2.5pt); %
    \filldraw (2,1) node[above=2] {\small$3$} circle (2.5pt); %
    \filldraw (0,0) node[above=2] {\small$4$} circle (2.5pt); %
    \filldraw (1,0) node[above=2] {\small$5$} circle (2.5pt); %
    \filldraw (2,0) node[above=2] {\small$6$} circle (2.5pt); %
  
    \draw[thick](0,0)--(2,0);%
    \draw[thick](0,1)--(2,1);%
    \node at (1,-0.75) {$M$};%

    \node[inner sep = 0.3mm] (em) at (4,-0.5) {\footnotesize $(0,0)$};%
    \node[inner sep = 0.3mm] (a) at (3.25,0.5) {\footnotesize
      $(3,2)$};%
    \node[inner sep = 0.3mm] (b) at (4.75,0.5) {\footnotesize
      $(3,2)$};%
    \node[inner sep = 0.3mm] (t) at (4,1.5) {\footnotesize $(6,3)$};%

    \foreach \from/\to in {em/a,em/b,a/t,b/t} \draw(\from)--(\to);%
    
    \filldraw (6,0.5) node[above=2] {\small$1$} circle (2.5pt); %
    \filldraw (7,0.75) node[above=2] {\small$2$} circle (2.5pt); %
    \filldraw (8,1) node[above=2] {\small$3$} circle (2.5pt); %
    \filldraw (7,0.25) node[below=2] {\small$5$} circle (2.5pt); %
    \filldraw (8,0) node[below=2] {\small$6$} circle (2.5pt); %
    \filldraw (8.5,0.5) node[above=2] {\small$4$} circle (2.5pt); %
    
    \draw[thick](6,0.5)--(8,0);%
    \draw[thick](6,0.5)--(8,1);%
    \node at (7,-0.75) {$M'$};%
  \end{tikzpicture}
  \caption{Two non-isomorphic matroids $M$ and $M'$ that have the same
    configuration.  Each pair shown in the lattice gives the size and
    rank of the corresponding cyclic flat.}
  \label{fig:sameconfig}
\end{figure}
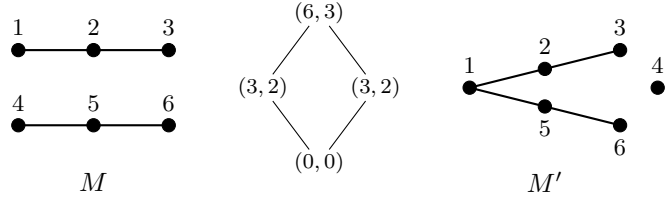

The next lemma holds since, for any element $x$ of the lattice in the
configuration, we can compare the size $s(x)$ and rank $\rho(x)$ to
those of pairs $y,z$ for which $y\meet z=0$ and $y\join z=x$.

\begin{lemma}\label{lem:configdirsum}
  For a cyclic flat $F$ of $M$, whether $M|F$ is connected can be
  deduced from the corresponding element of the configuration.  For
  each connected component $X$ of $M$, we can obtain the configuration
  of the restriction $M|X$ from that of $M$, so $M$ is configuration
  unique if and only if all such restrictions $M|X$ are configuration
  unique.
\end{lemma}

A pair $(X,Y)$ of sets in a matroid $M$ is \emph{modular} if
$r(X)+r(Y)=r(X\cup Y)+r(X\cap Y)$.  Routine calculations with the rank
function of the dual give the following two results.

\begin{lemma}\label{lem:modprcompdual}
  For subsets $X$ and $Y$ of $E(M)$, the pair $(X,Y)$ is modular in
  $M$ if and only if the pair $(E(M)-X,E(M)-Y)$ is modular in $M^*$.
\end{lemma}

\begin{lemma}\label{lem:compdualconfig}
  For any matroid $M$, the configuration of $M^*$ can be computed from
  that of $M$.  Thus, $M$ and $N$ have the same configuration if and
  only if $M^*$ and $N^*$ have the same configuration, so $M$ is
  configuration unique if and only if $M^*$ is configuration unique.
\end{lemma}

We next define principal extension, which makes precise the idea of
adding a point freely to a flat of a matroid.  For a matroid $M$, a
subset $X$ of $E(M)$, and an element $e\notin E(M)$, define
$r':2^{E(M)\cup e}\to \mathbb{Z}$ by, for all $Y\subseteq E(M)$,
setting $r'(Y)=r_M(Y)$ and
$$r'(Y\cup e)=
\begin{cases}
  r_M(Y), & \text{if } X\subseteq\cl_M(Y),\\
  r_M(Y)+1, & \text{otherwise.}
\end{cases}$$
It is routine to check that $r'$ is the rank function of a matroid on
$E(M)\cup e$.  This matroid is the \emph{principal extension} of $M$
in which $e$ has been \emph{added freely} to $X$, and is denoted
$M+_X e$.  Clearly $M+_X e = M+_{\cl(X)} e$ and $\cl(X)\cup e$ is a
flat of $M+_X e$.  Also, for $X,Y\subseteq E(M)$ and
$e,f\not\in E(M)$, we have $(M+_X e)+_Y f=(M+_Y f)+_X e$, so the order
in which we apply several principal extensions to subsets of $E(M)$
does not matter.  The \emph{free extension} of $M$ is $M+_{E(M)} e$.
Note that for  $X\subseteq E(M)$, the set $X\cup e$ is a 
circuit of $M+_{E(M)} e$ if and only if $X$ is a basis of $M$.

\subsection{Transversal matroids}

A \emph{set system} on a set $E$ is an indexed family of subsets of
$E$, which we write as $\mathcal{A}=(A_1,A_2,\ldots,A_r)$.  A set may
occur multiple times in $\mathcal{A}$.  A \emph{partial transversal}
of $\mathcal{A}$ is a subset $I$ of $E$ for which there is an
injection $\phi:I\rightarrow [r]$ with $x\in A_{\phi(x)}$ for all
$x\in I$.  \emph{Transversals} of $\mathcal{A}$ are partial
transversals of size $r$.  Edmonds and Fulkerson~\cite{ef} showed that
the partial transversals of a set system $\mathcal{A}$ on $E$ are the
independent sets of a matroid on $E$; we say that $\mathcal{A}$ is a
\emph{presentation} of this \emph{transversal matroid}
$M[\mathcal{A}]$.  A transversal matroid is \emph{fundamental} (or
\emph{principal}) if for some presentation $(A_1,\ldots,A_r)$ and each
$i\in [r]$, some element in $A_i$ is in no $A_j$ with
$j\in [r]-\{i\}$.  The following well-known results, especially
Corollary \ref{cor:bmot}, are relevant to our work (see \cite{bru}).

\begin{lemma}\label{lem:r}
  Any transversal matroid $M$ has a presentation with $r(M)$ sets.  If
  $M$ has no coloops, then each presentation of $M$ has exactly $r(M)$
  nonempty sets.
\end{lemma}

\begin{lemma}
  If $M$ is a transversal matroid, then so is $M|X$ for each
  $X\subseteq E(M)$.  If $(A_1,\ldots,A_r)$ is a presentation of $M$,
  then $(A_1\cap X,\ldots,A_r\cap X)$ is a presentation of $M|X$.
\end{lemma}

\begin{cor}\label{cor:bmot}
  If $(A_1,\ldots,A_r)$ is a presentation of $M$ and $X$ is any cyclic
  set of $M$, then $r(X)=|\{i\,:\,X\cap A_i\ne \emptyset\}|$.
\end{cor}

Brylawski \cite{aff} gave a useful way to view a transversal matroid
$M$.  Let $(A_1,A_2,\ldots,A_r)$ be a presentation of $M$ and let the
set $V=\{v_1,v_2,\ldots,v_r\}$ be disjoint from $E(M)$.  View the free
matroid on $V$ (i.e., all subsets of $V$ are independent) as having
the elements of $V$ at the vertices of an $r$-vertex simplex, one
element at each vertex.  For each $e\in E(M)$, add $e$ to this free
matroid by taking the principal extension using the set
$\{v_i\,:\,e\in A_i\}$; that is, put $e$ freely in the face of the
simplex that is spanned by the set $\{v_i\,:\,e\in A_i\}$ of vertices.
Once all elements of $E(M)$ are placed, delete $V$, and the result is
a geometric representation of $M$.  Note that each cyclic set of $M$
spans a face of the simplex.  Also, it follows that a transversal
matroid is fundamental if and only if it has a representation on a
simplex in which, for each vertex of the simplex, at least one element
of the matroid is placed there.

The characterization of transversal matroids in the next theorem was
first formulated by Mason~\cite{mason} using sets of cyclic sets; the
observation that his result easily implies its streamlined counterpart
for sets of cyclic flats was made by Ingleton~\cite{ing}.  For a
family $\mathcal{F}$ of sets we shorten $\cap_{A\in \mathcal{F}} A$ to
$\cap \mathcal{F}$ and $\cup_{A\in \mathcal{F}} A$ to
$\cup \mathcal{F}$.

\begin{thm}\label{thm:mi}
  A matroid is transversal if and only if for all nonempty sets
  $\mathcal{F}$ of cyclic flats,
  \begin{equation}\label{eq:mi}
    r(\cap \mathcal{F}) \leq
    \sum_{\mathcal{X}\subseteq \mathcal{F}}
    (-1)^{|\mathcal{X}|+1}r(\cup\mathcal{X}). 
  \end{equation}
\end{thm}

As explained in \cite{chartrans}, the condition in Theorem
\ref{thm:mi} is equivalent to having Inequality (\ref{eq:mi}) hold for
all nonempty antichains $\mathcal{F}$ of cyclic flats (i.e., no set in
$\mathcal{F}$ is a subset of another set in $\mathcal{F}$).
Inequality (\ref{eq:mi}) holds trivially when $|\mathcal{F}|=1$, and
it is the submodular inequality when $|\mathcal{F}|=2$.  Thus, to show
that a matroid is transversal, it suffices to check Inequality
(\ref{eq:mi}) for all antichains $\mathcal{F}$ of cyclic flats with
$|\mathcal{F}|\geq 3$,

\subsection{Lattice path matroids}

The lattice paths of interest are strings of steps that start at
$(0,0)$, and where each step has unit length and goes either north or
east.  We write lattice paths as words with the letters $N$ (north)
and $E$ (east).  Let $P=p_1p_2\dots p_n$ and $Q=q_1q_2\dots q_n$ be
two lattice paths from $(0,0)$ to $(m,r)$, so $m+r=n$, with $P$ never
going above $Q$. Let $p_{u_1},p_{u_2}, \ldots,p_{u_r}$ be the north
steps of $P$ with $u_1<u_2<\cdots<u_r$; let
$q_{l_1},q_{l_2}, \ldots,q_{l_r}$ be the north steps of $Q$ with
$l_1<l_2<\cdots<l_r$. Let $N_i$ be the interval $[l_i,u_i]$ of
integers. Let $M[P,Q]$ be the transversal matroid on the ground set
$[n]$ that has the presentation $(N_1,N_2,\ldots,N_r)$. In the example
in Figure \ref{fig:runningLPMex}, the upper path $Q$ is $NNENNENEE$
while $P$ is $EENENNENN$; the set $N_1$ is $\{1,2,3\}$, and, at the
top, the set $N_5$ is $\{7,8,9\}$.  A \emph{lattice path matroid} is a
matroid $M$ that is isomorphic to $M[P,Q]$ for some such pair of
lattice paths $P$ and $Q$.  In this paper, we will always take $E(M)$
to be $[n]$ with the usual order, and we will focus on the
presentation $(N_1,N_2,\ldots,N_r)$, which we call the \emph{path
  presentation} of $M$.

\begin{figure}
  \centering
  \begin{tikzpicture}[scale=0.6]
    \draw (0,0) grid (2,2); %
    \draw (1,1) grid (3,4); %
    \draw (2,3) grid (4,5);%

    \draw (-0.25,0.5) node {\footnotesize$1$};%
    \draw (0.75,0.5) node {\footnotesize$2$};%
    \draw (1.75,0.5) node {\footnotesize$3$};%

    \draw (-0.25,1.5) node {\footnotesize$2$};%
    \draw (0.75,1.5) node {\footnotesize$3$};%
    \draw (1.75,1.5) node {\footnotesize$4$};%
    \draw (2.75,1.5) node {\footnotesize$5$};%

    \draw (0.75,2.5) node {\footnotesize$4$};%
    \draw (1.75,2.5) node {\footnotesize$5$};%
    \draw (2.75,2.5) node {\footnotesize$6$};%

    \draw (0.75,3.5) node {\footnotesize$5$};%
    \draw (1.75,3.5) node {\footnotesize$6$};%
    \draw (2.75,3.5) node {\footnotesize$7$};%
    \draw (3.75,3.5) node {\footnotesize$8$};%

    \draw (1.75,4.5) node {\footnotesize$7$};%
    \draw (2.75,4.5) node {\footnotesize$8$};%
    \draw (3.75,4.5) node {\footnotesize$9$};%    
  \end{tikzpicture}
  \hspace{1cm}
  \begin{tikzpicture}[scale=1.3]
    \node[inner sep = 0.3mm] (em) at (0,0) {\footnotesize
      $\emptyset$};%
    \node[inner sep = 0.3mm] (1a) at (-1,0.75) {\footnotesize
      $\displaystyle\genfrac{}{}{0pt}{}{[3]}{2}$};%
    \node[inner sep = 0.3mm] (1b) at (1,0.75) {\footnotesize
      $\displaystyle\genfrac{}{}{0pt}{}{[7,9]}{2}$};%
    \node[inner sep = 0.3mm] (2a) at (-1.5,1.5) {\footnotesize
      $\displaystyle\genfrac{}{}{0pt}{}{[6]}{4}$};%
    \node[inner sep = 0.3mm] (2b) at (0,1.5) {\footnotesize
      $\displaystyle\genfrac{}{}{0pt}{}{[3]\cup[7,9]}{4}$};%
    \node[inner sep = 0.3mm] (2c) at (1.5,1.5) {\footnotesize
      $\displaystyle\genfrac{}{}{0pt}{}{[4,9]}{4}$};%
    \node[inner sep = 0.3mm] (3) at (0,2.35) {\footnotesize
      $\displaystyle\genfrac{}{}{0pt}{}{[9]}{5}$ };%

    \foreach \from/\to in {em/1a,em/1b,1a/2a,1a/2b,1b/2b,1b/2c,2a/3,
      2b/3,2c/3} \draw(\from)--(\to);%
\end{tikzpicture}  
\caption{The region bounded by two paths $P$ and $Q$ that give rise to
  a lattice path matroid, along with its lattice of cyclic flats, with
  the rank of each cyclic flat shown below the flat.  Each north step
  is labeled with the position it would have in any lattice path from
  $(0,0)$ to $(4,5)$ that uses that step.}
  \label{fig:runningLPMex}
\end{figure}
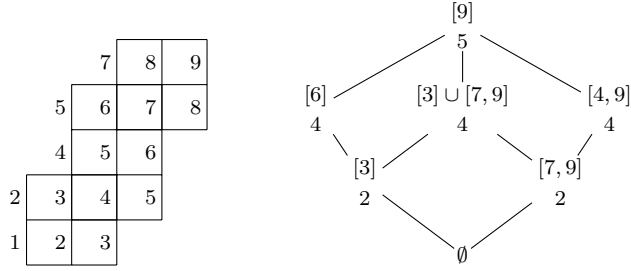

The transversals of the path presentation $(N_1,N_2,\ldots,N_r)$,
i.e., the bases of $M[P,Q]$, are the sets of positions of the north
steps in the lattice paths that go from $(0,0)$ to $(m,r)$ and remain
in the region that is bounded by $P$ and $Q$ (see \cite[Theorem
3.3]{lpm}).  It is easy to show that the lattice path matroid $M[P,Q]$
is connected if and only if the paths $P$ and $Q$ meet only at their
common endpoints, $(0,0)$ and $(m,r)$ (see \cite[Theorem 3.6]{lpm} and
the lattice path interpretation of direct sums discussed before it).

A \emph{nested matroid} is a lattice path matroid where either the
lower path $P$ has the form $E^{n-r}N^r$ or the upper path $Q$ has the
form $N^rE^{n-r}$.  When $P$ is $E^{n-r}N^r$ and $Q$ is $N^rE^{n-r}$,
the nested matroid is the uniform matroid $U_{r,n}$.  A matroid is
nested if and only if its lattice of cyclic flats is a chain (see
\cite[Lemma 2]{opr}; this also follows easily from the ideas in the
proof of Theorem \ref{thm:sameconfig} below).

The next result is from \cite[Lemma 5.2 and Theorem 5.7]{lpm2}.
\emph{Trivial flats} are the flats $X$ with $r(X)=|X|$.
\emph{Connected flats} are the flats $X$ for which $M|X$ is connected.
(In \cite{lpm2}, the term \emph{fundamental flats} refers to the flats
that satisfy the first property below.)

\begin{lemma}\label{thm:cflats}
  Let $M$ be a connected lattice path matroid on $[n]$ that is not a
  nested matroid.  There are two chains of proper, nontrivial,
  connected flats, $F_1\subsetneq F_2\subsetneq\cdots\subsetneq F_h$ and
  $G_1\subsetneq G_2\subsetneq\cdots\subsetneq G_k$, in $M$ that have the
  following properties:
  \begin{itemize}
  \item the flats in those chains are precisely the proper,
    nontrivial, connected flats $X$ for which, for some spanning
    circuit $C$ of $M$, the set $C\cap X$ is a basis of $M|X$,
  \item each $F_i$ is an interval $[a]$ and each $G_j$ is an interval
    $[b,n]$, and
  \item the other proper, nontrivial, connected flats of $M$ are the
    intersections $F_i\cap G_j$ for which
    $\eta(M)<\eta(F_i)+\eta(G_j)$, where $\eta(X)$ is the nullity of
    $X$, that is, $|X|-r(X)$.
  \end{itemize}
\end{lemma}

The flats $F_i$ of $M[P,Q]$ above are the intervals $[a]$ for which
steps $a$ and $a+1$ of the upper path $Q$ are east and north,
respectively.  The point at which such steps $a$ and $a+1$ of $Q$ meet
is an \emph{$EN$ corner of $Q$} and we call $F_i=[a]$ an \emph{initial
  connected flat}.  The flats $G_j$ are the intervals $[b,n]$ for
which steps $b-1$ and $b$ of the lower path $P$ are north and east,
respectively.  The point at which such steps $b-1$ and $b$ of $P$ meet
is an \emph{$NE$ corner of $P$} and we call $G_j=[b,n]$ a \emph{final
  connected flat}.  By Lemma \ref{thm:cflats}, each connected flat of
$M[P,Q]$ is an interval in $[n]$.

Observe that if the lattice path diagram for $M[P,Q]$ is rotated
$180^\circ$ about $(m/2,r/2)$, so step $i$ becomes step $n+1-i$, then
the original initial connected flats give rise to the final connected
flats after the rotation, and likewise with initial and final
switched.  The lower path $P$ gives the upper path $P^t$ after the
rotation, where the steps in $P^t$ are those of $P$ but in the reverse
order; likewise, the upper path $Q$ gives the lower path $Q^t$ after
the rotation.  The bijection mapping $i$ to $n+1-i$ is an isomorphism
of $M[P,Q]$ onto $M[Q^t,P^t]$.  If we know the size and rank of each
flat in the two chains identified in Theorem \ref{thm:cflats}, then we
know $P$ and $Q$ up to the $180^\circ$ rotation, and so we know
$M[P,Q]$ up to isomorphism.  (See \cite[Theorem 5.6]{lpm2} and the
discussion before it.)

The dual of a lattice path matroid $M$ is also a lattice path matroid;
its diagram is obtained by flipping the diagram for $M$ around the
line $y=x$.  This holds because, in any lattice path, this flip
switches the steps that are not in the corresponding basis (east
steps) with those that are in the basis (north steps).  A consequence
of this is \cite[Corollary 5.5]{lpm2}, which we state next.

\begin{lemma}\label{lem:complements}
  For a connected lattice path matroid $M$ on $[n]$, the interval
  $[a]$ is an initial connected flat of $M$ if and only if its
  complement $[a+1,n]$ is a final connected flat of the dual $M^*$.
  The same holds with $M$ and $M^*$ switched.
\end{lemma}

In addition to being closed under duality, the class of lattice path
matroids is also closed under direct sums and minors (see
\cite[Theorem 3.1]{lpm2}).  We will use just the special case of
restriction to an interval, the endpoints of which are not loops.
Given the lattice path diagram for $M[P,Q]$, to obtain the diagram for
the restriction of $M[P,Q]$ to an interval $[a,b]$, where neither $a$
nor $b$ is a loop of $M[P,Q]$, consider the lowest north step that can
be step $a$ in a path and the highest north step that can be step $b$
in a path; if the former is not strictly to the right of the latter,
then the restriction of $M[P,Q]$ to $[a,b]$ is represented by the
region of the diagram for $M[P,Q]$ that is between the two steps just
identified; otherwise the restriction is a free matroid.  This is
illustrated in Figure \ref{fig:runningLPMexdel} and it gives the
following lemma.

\begin{figure}
  \centering
  \begin{tikzpicture}[scale=0.6]
    \draw (0,0) grid (2,2); %
    \draw (1,1) grid (3,4); %
    \draw (2,3) grid (4,5);%

    \draw (-0.25,0.5) node {\footnotesize$1$};%
    \draw (0.75,0.5) node {\footnotesize$2$};%
    \draw (1.75,0.5) node {\footnotesize$3$};%

    \draw (-0.25,1.5) node {\footnotesize$2$};%
    \draw (0.75,1.5) node {\footnotesize$3$};%
    \draw (1.75,1.5) node {\footnotesize$4$};%
    \draw (2.75,1.5) node {\footnotesize$5$};%

    \draw (0.75,2.5) node {\footnotesize$4$};%
    \draw (1.75,2.5) node {\footnotesize$5$};%
    \draw (2.75,2.5) node {\footnotesize$6$};%

    \draw (0.75,3.5) node {\footnotesize$5$};%
    \draw (1.75,3.5) node {\footnotesize$6$};%
    \draw (2.75,3.5) node {\footnotesize$7$};%
    \draw (3.75,3.5) node {\footnotesize$8$};%

    \draw (1.75,4.5) node {\footnotesize$7$};%
    \draw (2.75,4.5) node {\footnotesize$8$};%
    \draw (3.75,4.5) node {\footnotesize$9$};%    
    \node at (2,-1) {$M$};%
  \end{tikzpicture}
  \hspace{1cm}
  \begin{tikzpicture}[scale=0.6]
    \draw (1,0) grid (2,2); %
    \draw (1,1) grid (3,4); %
    \draw (2,3) grid (3,5);%

    \draw (0.75,0.5) node {\footnotesize$2$};%
    \draw (1.75,0.5) node {\footnotesize$3$};%

    \draw (0.75,1.5) node {\footnotesize$3$};%
    \draw (1.75,1.5) node {\footnotesize$4$};%
    \draw (2.75,1.5) node {\footnotesize$5$};%

    \draw (0.75,2.5) node {\footnotesize$4$};%
    \draw (1.75,2.5) node {\footnotesize$5$};%
    \draw (2.75,2.5) node {\footnotesize$6$};%

    \draw (0.75,3.5) node {\footnotesize$5$};%
    \draw (1.75,3.5) node {\footnotesize$6$};%
    \draw (2.75,3.5) node {\footnotesize$7$};%

    \draw (1.75,4.5) node {\footnotesize$7$};%
    \draw (2.75,4.5) node {\footnotesize$8$};%
    \node at (2,-1) {$M|[2,8]$};%
  \end{tikzpicture}
  \hspace{1cm}
  \begin{tikzpicture}[scale=0.6]
    \draw (1,0) grid (2,4); %
    \draw (2,5) -- (2,4); %

    \draw (0.75,0.5) node {\footnotesize$2$};%
    \draw (1.75,0.5) node {\footnotesize$3$};%

    \draw (0.75,1.5) node {\footnotesize$3$};%
    \draw (1.75,1.5) node {\footnotesize$4$};%

    \draw (0.75,2.5) node {\footnotesize$4$};%
    \draw (1.75,2.5) node {\footnotesize$5$};%

    \draw (0.75,3.5) node {\footnotesize$5$};%
    \draw (1.75,3.5) node {\footnotesize$6$};%

    \draw (1.75,4.5) node {\footnotesize$7$};%

    \node at (1.5,-1) {$M|[2,7]$};%
  \end{tikzpicture}
  \caption{The diagrams representing a lattice path matroid $M$ and
    its restrictions to $[2,8]$ and $[2,7]$.  The restriction to
    $[3,5]$ is the free matroid $U_{3,3}$ on $[3,5]$, and the
    restriction to $[2,5]$ is also free.}
  \label{fig:runningLPMexdel}
\end{figure}
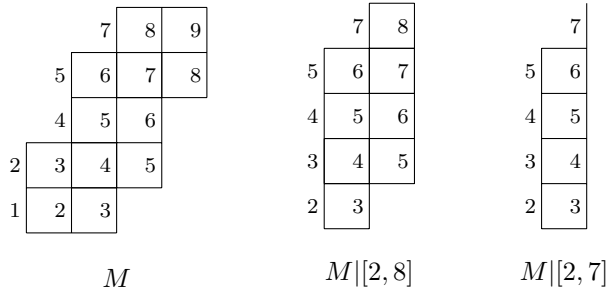

\begin{lemma}\label{lem:rankinLPM}
  Let $M$ be a lattice path matroid on $[n]$ and let the path
  presentation of $M$ be $(N_1,N_2,\ldots,N_r)$.  For an interval
  $X=[a+1,a+k]$ in $[n]$, if the sets $N_i$ with
  $N_i\cap X\ne \emptyset$ are $N_{j+1},N_{j+2},\ldots,N_{j+t}$, then
  $r(X) = \min(t,k)$.
\end{lemma}

\section{A construction to produce matroids having the same
  configuration as a given matroid}\label{sec:construction}

The next theorem gives the central construction of the paper.
Throughout the paper, we set
$\mathcal{Z}_e=\{A\in\mathcal{Z}(M)\,:\, e\in A\}$ for any
$e\in E(M)$.  To start with an example for motivation, consider the
matroid $M$ in Figure \ref{fig:sameconfig}, and the elements $1$ and
$4$.  We have $\mathcal{Z}_1=\{[3],[6]\}$ and
$\mathcal{Z}_4=\{[4,6],[6]\}$.  Both $\mathcal{Z}_1-\mathcal{Z}_4$ and
$\mathcal{Z}_4-\mathcal{Z}_1$ are nonempty, and the only pair $(A,B)$
with $A\in \mathcal{Z}_1-\mathcal{Z}_4$ and
$B\in \mathcal{Z}_4-\mathcal{Z}_1$ is non-modular.  We get the cyclic
flats of the matroid $M'$ in Figure \ref{fig:sameconfig}, which has
the same configuration as $M$, by taking the cyclic flat in
$\mathcal{Z}_4-\mathcal{Z}_1$, namely, $[4,6]$, and replacing $4$ by
$1$ to get $\{1,5,6\}$.

\begin{thm}\label{thm:twofilters}
  Let $M$ be a matroid.  Assume that for some $x,y\in E(M)$, (i) both
  $\mathcal{Z}_x-\mathcal{Z}_y$ and $\mathcal{Z}_y-\mathcal{Z}_x$ are
  nonempty, and (ii) if $X\in\mathcal{Z}_x-\mathcal{Z}_y$ and
  $Y\in\mathcal{Z}_y-\mathcal{Z}_x$, then $(X,Y)$ is not a modular
  pair.  For each $Y\in\mathcal{Z}_y-\mathcal{Z}_x$, let
  $Y_x=(Y-y)\cup x$, let
  $$\mathcal{Z}' =(\mathcal{Z}(M) -(\mathcal{Z}_y-\mathcal{Z}_x))\cup
  \{Y_x\,:\, Y\in \mathcal{Z}_y-\mathcal{Z}_x\},$$ and let
  $r':\mathcal{Z}'\to\mathbb{Z}$ be given by $r'(A)=r_M(A)$ if
  $A\in\mathcal{Z}(M)$, and $r'(Y_x)=r_M(Y)$ if
  $Y\in\mathcal{Z}_y-\mathcal{Z}_x$.  Then the pair
  $(\mathcal{Z}',r')$ satisfies properties (Z0)--(Z3) in Theorem
  \ref{thm:axioms} and so defines a matroid $M'$ on $E(M)$.  The
  matroids $M$ and $M'$ have the same configuration but are not
  isomorphic.  Also, the matroid $M'$ is isomorphic to the matroid
  that results from the construction above with the roles of $x$ and
  $y$ switched.
\end{thm}

\begin{proof}
  By the construction, the map $\phi: \mathcal{Z}(M)\to \mathcal{Z}'$
  defined by
  $$\phi(Y) =
  \begin{cases}
    Y_x, & \text{ if } Y\in\mathcal{Z}_y-\mathcal{Z}_x,\\
    Y, & \text{ otherwise }
  \end{cases}$$
  is a bijection.  The following properties of $\phi$
  are  easy to check: for all $X,Y\in \mathcal{Z}(M)$,  
  \begin{enumerate}
  \item[(i)] $X\subseteq Y$ if and only if $\phi(X)\subseteq \phi(Y)$,
  \item[(ii)] $|\phi(X)\cap \phi(Y)|= |X\cap Y|+1$ if
    $X\in\mathcal{Z}_x-\mathcal{Z}_y$ and
    $Y\in\mathcal{Z}_y-\mathcal{Z}_x$, or vice versa; otherwise
    $|\phi(X)\cap \phi(Y)|= |X\cap Y|$.
  \end{enumerate}
  By property (i), $\phi$ is a lattice isomorphism from
  $(\mathcal{Z}(M),\subseteq)$ onto $(\mathcal{Z}',\subseteq)$, so
  property (Z0) holds for the pair $(\mathcal{Z}',r')$.  Properties
  (Z1)--(Z3) are direct to check, with only slightly more care needed
  for property (Z3) for sets $X$ and $Y_x$ where
  $X\in \mathcal{Z}_x-\mathcal{Z}_y$ and
  $Y\in \mathcal{Z}_y-\mathcal{Z}_x$.  For that, we have
  $$r_M(X\cap Y) = r_M(X\meet Y)+|(X\cap Y)-(X\meet Y)|$$
  since the elements of $(X\cap Y)-(X\meet Y)$ are the coloops of
  $M|(X\cap Y)$, so the assumption that $(X,Y)$ is not a modular pair
  gives
  $$r_M(X\join Y) + r_M(X\meet Y) + |(X\cap Y) - (X\meet Y)|<
  r_M(X)+r_M(Y).$$ Since
  $|(X\cap Y_x) - (X\meet Y_x)|=|(X\cap Y) - (X\meet Y)|+1$ by
  property (ii), we have
  $$r'(X\join Y_x) + r'(X\meet Y_x) + |(X\cap Y_x)
  - (X\meet Y_x)|\leq r'(X)+r'(Y_x),$$ so property (Z3) holds for $X$
  and $Y_x$.  Thus, the pair $(\mathcal{Z}',r')$ indeed defines a
  matroid $M'$ on $E(M)$.  Clearly the configuration of $M'$ is that
  of $M$. Also, $M$ and $M'$ are not isomorphic since the multisets of
  sizes of intersections of cyclic flats differ by property (ii).

  By construction, the cyclic flats of $M'$ that contain just one of
  $x$ and $y$ must contain $x$. If we switch the roles of $x$ and $y$,
  then the cyclic flats of the resulting matroid $M''$ that contain
  just one of $x$ and $y$ must contain $y$.  Thus, the transposition
  of $E(M)$ that switches $x$ and $y$ and fixes all other elements of
  $E(M)$ is an isomorphism of $M'$ onto $M''$.
\end{proof}

\begin{figure}
  \centering
  \begin{tikzpicture}[scale=1]

    \draw[thick] (210:2)-- (320:1.54);%
    \draw[thick] (210:2)-- (100:1.54);%
    \draw[thick] (80:1.54)-- (-20:1.54);%
    
    \filldraw (210:2) node[above left] {\small$x$} circle (2.5pt); %
    \filldraw (320:1.54) node[below=2] {\small$d$} circle (2.5pt); %
    \filldraw (260:1) node[below=4] {\small$c$} circle (2.5pt); %
    \filldraw (100:1.54) node[left=2] {\small$b$} circle (2.5pt); %
    \filldraw (160:1) node[left=2] {\small$a$} circle (2.5pt); %
    \filldraw (80:1.54) node[right=2] {\small$y$} circle (2.5pt); %
    \filldraw (-20:1.54) node[right=2] {\small$f$} circle (2.5pt); %
    \filldraw (30:1) node[right=2] {\small$e$} circle (2.5pt); %
 
    \node at (0,-1.75) {$M$};%
  \end{tikzpicture}
  \hspace{1cm}
    \begin{tikzpicture}[scale=1]

    \draw[thick] (210:2)-- (320:1.54);%
    \draw[thick] (210:2)-- (100:1.54);%
    \draw[thick] (210:2)-- (210:-1);%
    
    \filldraw (210:2) node[above left] {\small$x$} circle (2.5pt); %
    \filldraw (320:1.54) node[below=2] {\small$d$} circle (2.5pt); %
    \filldraw (260:1) node[below=4] {\small$c$} circle (2.5pt); %
    \filldraw (100:1.54) node[left=2] {\small$b$} circle (2.5pt); %
    \filldraw (160:1) node[left=2] {\small$a$} circle (2.5pt); %
    \filldraw (70:1.5) node[right=2] {\small$y$} circle (2.5pt); %
    \filldraw (210:0.5) node[above=2] {\small$f$} circle (2.5pt); %
    \filldraw (30:1) node[right=2] {\small$e$} circle (2.5pt); %
 
    \node at (0,-1.75) {$M'$};%
  \end{tikzpicture}
  \caption{The matroid $M$ is transversal, but the matroid $M'$ that
    results when $x$ bumps $y$ is not transversal.  Both matroids have
    rank $3$.}
  \label{fig:doesnotpreservetrans}
\end{figure}

We say that the matroid $M'$ in Theorem \ref{thm:twofilters} results
from $x$ \emph{bumping} $y$ in $M$.  Bumping need not preserve the
property of being transversal, as the example in Figure
\ref{fig:doesnotpreservetrans} shows.  (One can apply either Theorem
\ref{thm:mi} or the geometric view of transversal matroids to verify
that $M$ is transversal and $M'$ is not.)  Likewise, the example in
Figure \ref{fig:sameconfig} shows that representability over a given
field need not be preserved; in that figure, $M$ is ternary but $M'$
is not.

When $q$ is a prime power that is at least $9$ and not prime, there
can be non-isomorphic projective planes of order $q$; many
constructions of such planes are known (see, e.g., \cite{PP}).  All
projective planes of order $q$ have the same configuration.  The
hypotheses of Theorem \ref{thm:twofilters} never hold for two elements
in a projective plane, so bumping does not apply to such matroids.
Also, bumping does not produce projective planes since no two elements
$x$ and $y$ in a projective plane have the property noted above, that
any cyclic flat that contains one of $x$ or $y$ must contain $x$.
Thus, bumping does not account for all instances of matroids that have
the same configuration.

If $(X,Y)$ is a non-modular pair of hyperplanes of a matroid $N$, then
extending $N$ by two principal extensions, one adding an element $x$
freely to $X$ and the other adding an element $y$ freely to $Y$, where
$x,y\not\in E(N)$, gives a matroid $M$ to which Theorem
\ref{thm:twofilters} applies.  Thus, such matroids are at most a
two-element extension away from a matroid that is not configuration
unique.

The next example, using the lattice path matroid in Figure
\ref{fig:runningLPMex}, is a preview of what we will see in Section
\ref{sec:lpm} for lattice path matroids that are not fundamental
transversal matroids.  The modular pairs that consist of an initial
connected flat and a final connected flat are $([6],[4,9])$ and
$([3],[7,9])$; the non-modular pairs of this type are $([6],[7,9])$
and $([3],[4,9])$.  Consider $3$ and $4$.  We have
$$\mathcal{Z}_3=\{[3],[6],[3]\cup[7,9],[9] \}\qquad
\text{ and } \qquad \mathcal{Z}_4=\{[6],[4,9],[9] \}.$$ Both
$\mathcal{Z}_3-\mathcal{Z}_4$ and $\mathcal{Z}_4-\mathcal{Z}_3$ are
nonempty; also, for each $X\in\mathcal{Z}_3-\mathcal{Z}_4$ and
$Y\in\mathcal{Z}_4-\mathcal{Z}_3$, the pair $(X,Y)$ is not modular, so
the hypotheses of Theorem \ref{thm:twofilters} hold.  The result of
replacing $4$ by $3$ in the sole cyclic flat in
$\mathcal{Z}_4-\mathcal{Z}_3$ is shown in Figure \ref{fig:twooptions}.

\begin{figure}
  \centering
  \begin{tikzpicture}[scale=1.3]
    \node[inner sep = 0.3mm] (em) at (0,0) {\footnotesize
      $\emptyset$};%
    \node[inner sep = 0.3mm] (1a) at (-1,0.75) {\footnotesize
      $\displaystyle\genfrac{}{}{0pt}{}{[3]}{2}$};%
    \node[inner sep = 0.3mm] (1b) at (1,0.75) {\footnotesize
      $\displaystyle\genfrac{}{}{0pt}{}{[7,9]}{2}$};%
    \node[inner sep = 0.3mm] (2a) at (-1.75,1.5) {\footnotesize
      $\displaystyle\genfrac{}{}{0pt}{}{[6]}{4}$};%
    \node[inner sep = 0.3mm] (2b) at (0,1.5) {\footnotesize
      $\displaystyle\genfrac{}{}{0pt}{}{[3]\cup[7,9]}{4}$};%
    \node[inner sep = 0.3mm] (2c) at (1.75,1.5) {\footnotesize
      $\displaystyle\genfrac{}{}{0pt}{}{{\mathbf{\{3,5,6,7,8,9\}}}}
      {4}$};%
    \node[inner sep = 0.3mm] (3) at (0,2.35) {\footnotesize
      $\displaystyle\genfrac{}{}{0pt}{}{[9]}{5}$ };%

    \foreach \from/\to in {em/1a,em/1b,1a/2a,1a/2b,1b/2b,1b/2c,2a/3,
      2b/3,2c/3} \draw(\from)--(\to);%
  \end{tikzpicture}
  \caption{The lattice of cyclic flats obtained when $3$ bumps $4$ in
    the matroid shown in Figure \ref{fig:runningLPMex}.  In this case,
    only one set is altered; it is highlighted with boldface.}
  \label{fig:twooptions}
\end{figure}
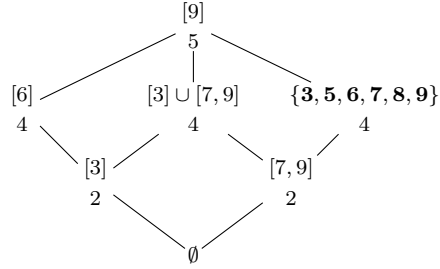

We end this section with an immediate corollary of Theorem
\ref{thm:twofilters}, Lemma \ref{lem:modprcompdual}, and the equality
$\mathcal{Z}(M^*)=\{E(M)-A\,:\, A\in \mathcal{Z}(M)\}$: up to
switching $x$ and $y$, bumping commutes with taking the dual.

\begin{cor}\label{cor:dualbump}
  For elements $x$ and $y$, the hypotheses of Theorem
  \ref{thm:twofilters} hold in $M$ if and only if they hold in $M^*$,
  and the matroid obtained from $x$ bumping $y$ in $M$ is the dual of
  the matroid obtained from $y$ bumping $x$ in $M^*$.
\end{cor}

\section{An application to lattice path matroids}\label{sec:lpm}

There are three main results in this section.  Theorem
\ref{thm:lpmonuniqueconfig} shows that bumping, the construction in
Theorem \ref{thm:twofilters}, applies to any connected lattice path
matroid that has a non-modular pair $(A,B)$ that consists of an
initial connected flat $A$ and a final connected flat $B$.  Theorem
\ref{thm:nomixingconfigunique} proves the converse by showing that
connected lattice path matroids in which all such pairs $(A,B)$ are
modular are configuration unique.  Corollary \ref{cor:rookfund} shows
that, for connected lattice path matroids, having all such pairs
$(A,B)$ be modular is equivalent to the matroid being fundamental
transversal.

We first characterize modular pairs of initial and final connected
flats.

\begin{lemma}\label{lem:modularpairsinLPM}
  Let $M$ be a connected lattice path matroid on $[n]$, and let
  $(N_1,N_2,\ldots,N_r)$ be its path presentation.  Let $A=[a]$ be an
  initial connected flat, and $B=[b,n]$ be a final connected flat, of
  $M$.  The pair $(A,B)$ is modular if and only if the number of
  $i\in[r]$ with both $A\cap N_i$ and $B\cap N_i$ nonempty is at most
  $|A\cap B|$.
\end{lemma}

\begin{proof}
  The assertion when $A\cap B=\emptyset$ follows from Corollary
  \ref{cor:bmot} applied to $A$, $B$, and $A\cup B$.  Now assume that
  $A\cap B\ne\emptyset$, so $b\leq a$ and $A\cap B = [b,a]$.  Since
  all sets involved are intervals, $A\cap N_i\ne \emptyset$ and
  $B\cap N_i\ne \emptyset$ if and only if
  $[b,a]\cap N_i\ne \emptyset$.  Thus, we must show that $(A,B)$ is
  modular if and only if
  $|\{i\,:\,[b,a]\cap N_i\ne \emptyset\}|\leq |A\cap B|$.  Let
  $\{i\,:\,[b,a]\cap N_i\ne \emptyset\} = [j+1,j+t]$.  Corollary
  \ref{cor:bmot} gives $r(A)=j+t$ and $r(B)=r-j$, so $(A,B)$ is
  modular if and only if $r(A\cap B)=t$, which, by Lemma
  \ref{lem:rankinLPM}, is precisely when $t\leq |A\cap B|$, as needed.
\end{proof}

We next recast Lemma \ref{lem:modularpairsinLPM} in terms of lattice
path diagrams.  Recall that the initial connected flats $[a]$ of
$M[P,Q]$ arise precisely from the $EN$ corners of the upper path $Q$,
where that east step is the $a$th step in $Q$, and the final connected
flats $[b,n]$ arise precisely from the $NE$ corners of $P$, where that
east step is the $b$th step in $P$.  For an initial connected flat
$A=[a]$, let $(a_1,a_2)$ be the coordinates of the integer point at
the corresponding $EN$ corner of $Q$.  Thus, $a_1+a_2=a$.  For a final
connected flat $B=[b,n]$, let $(b_1,b_2)$ be the coordinates of the
integer point at the corresponding $NE$ corner of $P$.  Thus,
$b_1+b_2=b-1$.  We call the pair $(A,B)$ \emph{mixed} if $a_1< b_1$
and $a_2> b_2$.  Since $P$ never goes above $Q$, we cannot have
$a_1> b_1$ and $a_2< b_2$, so the mixed case is the only option for
having the signs of $a_1-b_1$ and $a_2-b_2$ differ.

\begin{cor}\label{cor:modnotmixed}
  Let $M=M[P,Q]$ be a connected lattice path matroid on $[n]$.  Let
  $A$ be an initial connected flat of $M$ and let $B$ be a final
  connected flat of $M$.  The pair $(A,B)$ is modular if and only if
  it is not mixed.
\end{cor}

\begin{proof}
  Let $A$ be $[a]$ where $a$ corresponds to the $EN$ corner of $Q$ at
  $(a_1,a_2)$, and let $B$ be $[b,n]$ where $b$ corresponds to the
  $NE$ corner of $P$ at $(b_1,b_2)$.

  Assume that $(A,B)$ is not mixed.  First assume that $a_1\leq b_1$
  and $a_2\leq b_2$. Thus, $a<b$, so $A\cap B=\emptyset$.  Let $N_i$
  be a set in the path presentation of $M$.  If
  $N_i\cap A\ne \emptyset$, then $i\leq a_2$, while if
  $N_i\cap B\ne\emptyset$, then $i> b_2$.  Since $a_2\leq b_2$, no set
  $N_i$ satisfies both conditions, so $(A,B)$ is a modular pair by
  Lemma \ref{lem:modularpairsinLPM}.  Now assume that $a_1\geq b_1$
  and $a_2\geq b_2$.  In the lattice path diagram for the dual $M^*$,
  the upper bounding path has an $EN$ corner at $(b_2,b_1)$ and the
  lower bounding path has an $NE$ corner at $(a_2,a_1)$. By what we
  just proved, the pair $([n]-B,[n]-A)$, which consists of an initial
  and a final connected flat of $M^*$ by Lemma \ref{lem:complements},
  is modular in $M^*$.  Thus, $(A,B)$ is a modular pair in $M$ by
  Lemma \ref{lem:modprcompdual}.
  
  Now assume that $(A,B)$ is mixed, so $a_1< b_1$ and $a_2> b_2$.
  First assume that $A\cap B=\emptyset$, so $a<b$.  Now
  $b\in N_{b_2+1}$ and $a\in N_{a_2}$, so $\{a,b\}\subseteq N_i$ for
  all sets $N_i$ for which $b_2< i\leq a_2$.  Thus, $(A,B)$
  is not a modular pair.  Now assume that $A\cap B\ne \emptyset$.
  Then $([n]-B,[n]-A)$ is a mixed pair consisting of an initial and a
  final connected flat of $M^*$, and the sets are disjoint.  By what
  we just proved, the pair $([n]-B,[n]-A)$ is not modular in $M^*$, so
  the pair $(A,B)$ is not modular in $M$ by Lemma
  \ref{lem:modprcompdual}.
\end{proof}

The next theorem is the first main result of this section.  The
theorem is stated for connected lattice path matroids, but it extends
to any lattice path matroid by applying the result to the restrictions
to connected components.

\begin{thm}\label{thm:lpmonuniqueconfig}
  If a connected lattice path matroid $M$ on $[n]$ has an initial
  connected flat $A=[a]$ and a final connected flat $B=[b,n]$ for
  which $(A,B)$ is not a modular pair, then some matroid that is not
  isomorphic to $M$ has the same configuration as $M$.
\end{thm}

\begin{proof}
  Let $(N_1,N_2,\ldots,N_r)$ be the path presentation of $M$.  We
  first consider the case in which $A\cap B = \emptyset$, so $a<b$.
  Among all non-modular pairs of disjoint initial and final connected
  flats, choose $(A,B)$ so that $b-a$ is minimal.  We claim that
  $\mathcal{Z}_a$ and $\mathcal{Z}_b$ satisfy the hypotheses of
  Theorem \ref{thm:twofilters}.  Since
  $A\in \mathcal{Z}_a- \mathcal{Z}_b$ and
  $B\in \mathcal{Z}_b- \mathcal{Z}_a$, neither difference is empty.
  By Lemma \ref{lem:modularpairsinLPM}, some set $N_i$ contains both
  $a$ and $b$.  For any initial connected flat
  $[c]\in \mathcal{Z}_a- \mathcal{Z}_b$, we have $a\leq c<b$, so
  $([c],B)$ is not a modular pair since $[c]\cap B=\emptyset$ and
  $\{c,b\}\subseteq N_i$.  By having chosen $A$ and $B$ with $b-a$
  minimal, it follows that $c=a$, so $A$ is the only initial connected
  flat in $\mathcal{Z}_a- \mathcal{Z}_b$.  Similarly, $B$ is the only
  final connected flat in $\mathcal{Z}_b- \mathcal{Z}_a$.  By those
  conclusions and Lemma \ref{thm:cflats}, no flat in
  $\mathcal{Z}_a- \mathcal{Z}_b$ or $\mathcal{Z}_b- \mathcal{Z}_a$
  contains an element $c$ with $a<c<b$.  Consider
  $F_a\in \mathcal{Z}_a- \mathcal{Z}_b$ and
  $F_b\in \mathcal{Z}_b- \mathcal{Z}_a$.  Each connected component of
  $M|F_a$ or of $M|F_b$ is a subset of either $A$ or $B$. The sets
  $F_a$, $F_b$, and $F_a\cup F_b$ are cyclic, so the rank of each set
  is the number of sets $N_j$ that are not disjoint from it.  The set
  $F_a\cap F_b$ might not be cyclic, so the number of sets $N_j$ that
  are not disjoint from it is only an upper bound on its rank.  If
  $F_a\cap F_b=\emptyset$, then any set $N_i$ that contains $a$ and
  $b$ shows that $r(F_a)+r(F_b)-r(F_a\cup F_b)\geq 1>r(F_a\cap F_b)$,
  so $(F_a,F_b)$ is not a modular pair.  Now assume that
  $F_a\cap F_b\ne\emptyset$.  Fix $c\in F_a\cap F_b$.  Either $c<a$ or
  $c>b$.  First, assume that $c<a$.  Let $X_b$ be the connected
  component of $M|F_b$ that contains $b$, and let $X_c$ be the
  connected component of $M|F_b$ that contains $c$, so
  $X_b\subseteq B$ and $X_c\subseteq A$.  Since
  $X_b\cap X_c=\emptyset$ and $r(X_b\cup X_c)=r(X_b)+r(X_c)$, no set
  $N_j$ in the presentation contains both $c$ and $b$ (otherwise $N_j$
  would be counted twice on the right side and only once on the left
  side).  By symmetry, if $c>b$, then no $N_j$ contains both $c$ and
  $a$.  Thus, the sets $N_i$ that contain both $a$ and $b$ are
  disjoint from $F_a\cap F_b$ and so do not contribute to
  $r(F_a\cap F_b)$; however, they contribute to each of $r(F_a)$,
  $r(F_b)$, and $r(F_a\cup F_b)$.  Therefore
  $r(F_a)+r(F_b)-r(F_a\cup F_b)>r(F_a\cap F_b)$, so $(F_a,F_b)$ is not
  a modular pair.  Thus, $\mathcal{Z}_a$ and $\mathcal{Z}_b$ satisfy
  the hypotheses of Theorem \ref{thm:twofilters}, so the matroid that
  arises from $M$ when $a$ bumps $b$ is not isomorphic to $M$ and has
  the same configuration as $M$.

  Finally, assume that $A\cap B\ne \emptyset$, so $b\leq a$. By Lemma
  \ref{lem:complements}, the set $[b-1]$ is an initial connected flat
  of $M^*$, and $[a+1,n]$ is a final connected flat of $M^*$.  Also,
  the pair $([b-1],[a+1,n])$ is not modular in $M^*$ by Lemma
  \ref{lem:modprcompdual}, and $[b-1]\cap [a+1,n]=\emptyset$.  By the
  case shown above, bumping applies to some elements $a'$ and $b'$ in
  $M^*$, and so it applies to $a'$ and $b'$ in $M$ by Corollary
  \ref{cor:dualbump}, as needed.
\end{proof}

The matroid that is constructed in the proof of Theorem
\ref{thm:lpmonuniqueconfig} is not a lattice path matroid by the next
result.

\begin{thm}\label{thm:reconLPM}
  If two lattice path matroids have the same configuration, then they
  are isomorphic.
\end{thm}

\begin{proof}
  Let $(L, s, \rho,n)$ be the configuration of a lattice path matroid
  $M$ on $[n]$.  It suffices to show how to obtain $M$ up to
  isomorphism from $(L, s, \rho,n)$.  By Lemma \ref{lem:configdirsum},
  we may assume that $M$ is connected.  Nested matroids are Tutte
  unique \cite[Theorem 8.12]{AnnaThesis}, which is a stronger
  conclusion, so we can assume that $M$ is not nested.  By the
  comments two paragraphs after Lemma \ref{thm:cflats}, it suffices to
  identify the elements of $L$ that correspond to the connected flats
  in the two chains identified in that lemma.  We do this by showing
  how to identify all other elements of $L$.  Disconnected cyclic
  flats of $M$ can be detected from $(L, s, \rho,n)$ by Lemma
  \ref{lem:configdirsum}.  Let $F\in \mathcal{Z}(M)$ be connected but
  in neither chain identified in Lemma \ref{thm:cflats}, so $F$ is an
  intersection of an initial and a final connected flat.  Thus,
  $[n]-F$ is not an interval, so it is a disconnected cyclic flat of
  the dual $M^*$, and this can be detected from the configuration of
  $M^*$, which we get from $(L, s, \rho,n)$ by Lemma
  \ref{lem:compdualconfig}.
\end{proof}
  
The next result, which is another main result of this section,
strengthens Theorem \ref{thm:reconLPM} when each pair that consists of
an initial and a final connected flat is modular.  It has not yet been
shown whether these matroids are $\mathcal{G}$ unique.

\begin{thm}\label{thm:nomixingconfigunique}
  Let $M=M[P,Q]$ be a lattice path matroid.  If, for all initial
  connected flats $A$ and final connected flats $B$ of the restriction
  to any connected component of $M$, the pair $(A,B)$ is modular, then
  $M$ is configuration unique.
\end{thm}

\begin{proof}
  By Lemma \ref{lem:configdirsum}, we may assume that $M$ is
  connected.  Assume that $M$ and $N$ have the same configuration.
  Thus, there is a lattice isomorphism
  $\Phi:\mathcal{Z}(M)\to\mathcal{Z}(N)$ that preserves the size and
  rank of each cyclic flat.  We will shorten $\Phi(A)$ to $A'$.  To
  show that $M$ and $N$ are isomorphic, it suffices to show that
  $\Phi$ is induced by a bijection $\phi:E(M)\to E(N)$.  We first show
  that $|A\cap B|=|A'\cap B'|$ and $|A\cup B|=|A'\cup B'|$ for any
  initial connected flat $A$ and final connected flat $B$ of
  $M$. Showing one of those equalities suffices since each equality
  implies the other by inclusion/exclusion.

  Let $E(M)=[n]$, let $A$ be $[a]$ where $a$ corresponds to the $EN$
  corner of $Q$ at $(a_1,a_2)$, and let $B$ be $[b,n]$ where $b$
  corresponds to the $NE$ corner of $P$ at $(b_1,b_2)$.  Thus,
  $a_1+a_2=a$ and $b_1+b_2=b-1$.  The assumption that $(A,B)$ is
  modular means that $(A,B)$ is not mixed, which we break into two
  cases: (i) $a_1<b_1$ and $a_2\leq b_2$, and (ii) $b_1\leq a_1$ and
  $b_2<a_2$.
  
  Case (i) gives $a<b$, so $A\cap B=\emptyset$.  Now
  $r_M(A)+r_M(B)=r_M(A\join B)$ since $(A,B)$ is modular.  By applying
  $\Phi$ we get $r_N(A')+r_N(B')=r_N(A'\join B')$, which gives
  $A'\cap B'=\emptyset$ since otherwise the submodular inequality
  would fail for $A'$ and $B'$.  Thus, $|A\cap B|=0=|A'\cap B'|$.

  Assume that case (ii) applies to $(A,B)$.  Define
  $\Phi^*:\mathcal{Z}(M^*)\to \mathcal{Z}(N^*)$ by, for all $F$ in
  $\mathcal{Z}(M)$, setting $\Phi^*([n]-F)=E(N)-F'$.  The discussion
  of duality after Theorem \ref{thm:axioms} shows that $\Phi^*$ is a
  lattice isomorphism; also, it preserves size and rank.  By Lemma
  \ref{lem:complements}, the pair $([n]-B,[n]-A)$ consists of an
  initial and a final connected flat of $M^*$; also, the corresponding
  corners are at $(b_2,b_1)$ and $(a_2,a_1)$ in the lattice path
  diagram for $M^*$.  Case (i) applies to $([n]-B,[n]-A)$ in $M^*$, so
  by what we just proved,
  $$([n]-A)\cap ([n]-B)=\emptyset=(E(N)-A')\cap (E(N)- B').$$
  Thus, $A\cup B=[n]$ and $A'\cup B'=E(N)$, so
  $|A\cup B|=n=|A'\cup B'|$, as claimed.

  Let $A_1\subsetneq A_2\subsetneq \cdots\subsetneq A_{s-1}$ be the
  initial connected flats of $M$.  Set $A_0=\emptyset$ and $A_s=[n]$.
  Let $B_1\subsetneq B_2\subsetneq \cdots\subsetneq B_{t-1}$ be the
  final connected flats of $M$.  Set $B_0=\emptyset$ and $B_t=[n]$.
  Thus, $A'_0=B'_0= \emptyset$ and $A'_s=B'_s= E(N)$.  For each
  $e\in [n]$, we have $e\in(A_i-A_{i-1})\cap (B_j-B_{j-1})$ for
  exactly one pair $(i,j)\in[s]\times[t]$.  What we showed in the
  previous two paragraphs also gives
  $$|(A_i-A_{i-1})\cap (B_j-B_{j-1})|=|(A'_i-A'_{i-1})\cap
  (B'_j-B'_{j-1})|$$ for all $(i,j)\in[s]\times[t]$.
  Thus, there is a bijection   $\phi:E(M)\to E(N)$ for which 
  $$\phi\bigl((A_i-A_{i-1})\cap (B_j-B_{j-1})\bigr)=(A'_i-A'_{i-1})\cap
  (B'_j-B'_{j-1})$$ for all $(i,j)\in[s]\times[t]$.  It follows that
  $\phi(A_i)=A'_i$, $\phi(B_j)=B'_j$, and
  $\phi(A_i\cap B_j)=A'_i\cap B'_j$ for each $i\in[s-1]$ and
  $j\in[t-1]$.  By Lemma \ref{thm:cflats}, any connected component of
  any cyclic flat of $M$ and $N$ is among these sets.  Thus, $\phi$ is
  an isomorphism from $M$ onto $N$.
\end{proof}

In order to show that the condition on a lattice path matroid $M$ in
Theorem \ref{thm:lpmonuniqueconfig} holds if and only if $M$ is not a
fundamental transversal matroid, we will use rook matroids, which
Alexandersson and Jal introduced in \cite{rook}.  For consistency with
the convention for lattice path matroids, our description of rook
matroids differs superficially from \cite{rook}: we switch the roles
of rows and columns, and our labeling differs.  Given lattice paths
$P$ and $Q$ from $(0,0)$ to $(m,r)$ with $P$ never rising above $Q$,
label the rows of the diagram, from bottom to top, by $1$ through $r$,
and the columns, from left to right, by $r+1$ to $r+m$.  (Figure
\ref{fig:rook} gives an example.)  For each $i\in [r]$, let the set
$A_i$ consist of $i$ along with the labels of all columns that have a
square in row $i$.  The \emph{rook matroid} $R[P,Q]$ is the
transversal matroid with the presentation $(A_1,A_2,\ldots,A_r)$.  By
construction, $R[P,Q]$ is a fundamental transversal matroid; the
element $i\in[r]$ is in $A_i$ and in no other $A_j$.  In contrast,
lattice path matroids need not be fundamental.  In \cite{rook},
Alexandersson and Jal show that the bases of the rook matroid
correspond to non-attacking, non-nesting placements of rooks on the
board given by the lattice path diagram.  (With our labeling, rook
placements amount to bijections $\phi:C\to R$ where
$C\subseteq [r+1,r+m]$, $R\subseteq [r]$, and $c\in A_{\phi(c)}$ for
all $c\in C$; rooks are placed in the squares $(\phi(c),c)$, for
$c\in C$, on the board.  The non-nesting condition means that for
$c,c'\in C$, if $c<c'$, then $\phi(c)<\phi(c')$.)

\begin{figure}
  \centering
  \begin{tikzpicture}[scale=0.6]
    \draw (0,0) grid (2,2); %
    \draw (1,1) grid (3,4); %
    \draw (2,3) grid (4,5);%

    \draw (-0.35,0.5) node {\footnotesize$1$};%
    \draw (-0.35,1.5) node {\footnotesize$2$};%
    \draw (-0.35,2.5) node {\footnotesize$3$};%
    \draw (-0.35,3.5) node {\footnotesize$4$};%
    \draw (-0.35,4.5) node {\footnotesize$5$};%

    \draw (0.5,5.35) node {\footnotesize$6$};%
    \draw (1.5,5.35) node {\footnotesize$7$};%
    \draw (2.5,5.35) node {\footnotesize$8$};%
    \draw (3.5,5.35) node {\footnotesize$9$};%    
  \end{tikzpicture}
  \hspace{1cm}
  \begin{tikzpicture}[scale=1.3]
    \node[inner sep = 0.3mm] (em) at (0,0) {\footnotesize
      $\emptyset$};%
    \node[inner sep = 0.3mm] (1a) at (-1,0.75) {\footnotesize
      $\displaystyle\genfrac{}{}{0pt}{}{\{1,2,6\}}{2}$};%
    \node[inner sep = 0.3mm] (1b) at (1,0.75) {\footnotesize
      $\displaystyle\genfrac{}{}{0pt}{}{\{4,5,9\}}{2}$};%
    \node[inner sep = 0.3mm] (2a) at (-1.7,1.5) {\footnotesize
      $\displaystyle\genfrac{}{}{0pt}{}{\{1,2,3,4,6,7\} }{4}$};%
    \node[inner sep = 0.3mm] (2b) at (0,1.5) {\footnotesize
      $\displaystyle\genfrac{}{}{0pt}{}{\{1,2,4,5,6,9\}}{4}$};%
    \node[inner sep = 0.3mm] (2c) at (1.7,1.5) {\footnotesize
      $\displaystyle\genfrac{}{}{0pt}{}{\{2,3,4,5,8,9\}}{4}$};%
    \node[inner sep = 0.3mm] (3) at (0,2.35) {\footnotesize
      $\displaystyle\genfrac{}{}{0pt}{}{[9]}{5}$ };%

    \foreach \from/\to in {em/1a,em/1b,1a/2a,1a/2b,1b/2b,1b/2c,2a/3,
      2b/3,2c/3} \draw(\from)--(\to);%
\end{tikzpicture}  
\caption{The rook matroid for which the sets in the presentation are
  $\{1,6,7\}$, $\{2,6,7,8\}$, $\{3,7,8\}$, $\{4,7,8,9\}$, and
  $\{5,8,9\}$.  The lattice of cyclic flats is shown on the right.
  Note the correspondence with the cyclic flats of the matroid in
  Figure \ref{fig:runningLPMex}.}
  \label{fig:rook}
\end{figure}
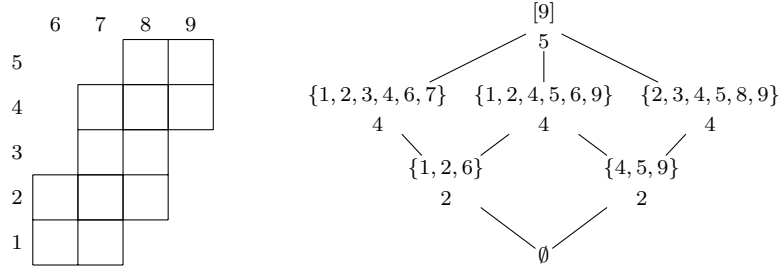

The interpretation of direct sums via diagrams is the same for rook
matroids as for lattice path matroids, so, like lattice path matroids,
the class of rook matroids is closed under direct sums; also, a rook
matroid is connected if and only if the bounding paths never meet
except at the first and last points.  (In contrast, as one would expect
for a class of fundamental transversal matroids, the class of rook
matroids is not closed under minors.)  Even if the rook matroid
$R[P,Q]$ and the lattice path matroid $M[P,Q]$ are not isomorphic,
they have the same configuration, as we show in Theorem
\ref{thm:sameconfig}.

We first treat a lemma that applies to both lattice path and rook
matroids.  Given a presentation $\mathcal{Y}=(Y_1,Y_2,\ldots,Y_r)$ of
a transversal matroid $N$ and an element $e\in E(N)$, the
\emph{support of} $e$, denoted $s_{\mathcal{Y}}(e)$, is
$\{i\in [r]\,:\,e\in Y_i\}$.  The \emph{support $s_{\mathcal{Y}}(X)$
  of} $X\subseteq E(M)$ is the union of all sets $s_{\mathcal{Y}}(e)$
with $e\in X$.  By Hall's theorem, $I\subseteq E(N)$ is independent in
$N$ if and only if $|X|\leq|s_{\mathcal{Y}}(X)|$ for all subsets $X$
of $I$, so $C\subseteq E(N)$ is a circuit of $N$ if and only if
$|s_{\mathcal{Y}}(C)|<|C|$ but $|X|\leq|s_{\mathcal{Y}}(X)|$ for all
proper subsets $X$ of $C$.

\begin{lemma}\label{lem:suppint}
  Let $\mathcal{Y}=(Y_1,Y_2,\ldots,Y_r)$ be a presentation of a
  transversal matroid $N$.  If, for each $e\in E(N)$, the support
  $s_{\mathcal{Y}}(e)$ is an interval in $[r]$, then
  \begin{enumerate}
  \item for any circuit $C$ of $N$, its support $s_{\mathcal{Y}}(C)$
    is an interval in $[r]$,
  \item the support $s_{\mathcal{Y}}(F)$ of any connected flat $F$ of
    $N$ with $|F|>1$ is an interval, say $I$, and
    $F=\{e\,:\, s_{\mathcal{Y}}(e)\subseteq I\}$.
  \end{enumerate}
\end{lemma}

\begin{proof}
  Let $C$ be a circuit of $N$.  As noted above,
  $|C|> |s_{\mathcal{Y}}(C)|$, while $|X|\leq |s_{\mathcal{Y}}(X)|$
  for all $X\subsetneq C$.  Assume, contrary to part (1), that
  $s_{\mathcal{Y}}(C)$ is not an interval, so there is a partition
  $\{I,J\}$ of $s_{\mathcal{Y}}(C)$ where $I$ is a maximal interval in
  $s_{\mathcal{Y}}(C) $ and $J=s_{\mathcal{Y}}(C)-I$.  Let
  $C_I=\{e\in C\,:\, s_{\mathcal{Y}}(e)\subseteq I\}$ and
  $C_J=\{e\in C\,:\, s_{\mathcal{Y}}(e)\subseteq J\}$.  Since each set
  $s_{\mathcal{Y}}(e)$ is an interval, $\{C_I,C_J\}$ is a partition of
  $C$.  The inequality $|s_{\mathcal{Y}}(C)|<|C|$ implies that either
  $|s_{\mathcal{Y}}(C_I)|<|C_I|$ or $|s_{\mathcal{Y}}(C_J)|<|C_J|$,
  which contradicts having $|X|\leq |s_{\mathcal{Y}}(X)|$ for all
  $X\subsetneq C$, and so proves assertion (1).

  Any two elements in a connected flat $F$ are in a circuit of $N|F$,
  so it follows from part (1) that $s_{\mathcal{Y}}(F)$ is an interval
  $I$ in $[r]$.  Corollary \ref{cor:bmot} gives $r(F)=|I|$, so if
  $s_{\mathcal{Y}}(e)\subseteq I$, then $r(F)=r(F\cup e)$.  Thus,
  $\{e\,:\, s_{\mathcal{Y}}(e)\subseteq I\} \subseteq F$.  Clearly
  $F\subseteq\{e\,:\, s_{\mathcal{Y}}(e)\subseteq I\}$, so equality
  holds.
\end{proof}

For example, consider the lattice path matroid $M$ in Figure
\ref{fig:runningLPMex} and the rook matroid $R$ in Figure
\ref{fig:rook}.  In $M$, the connected cyclic flats are (denoting the
support of $e$ by $s_{\mathcal{N}}(e)$)
\begin{align*}
  &
    \begin{aligned}[t]
      \emptyset=&\,\{e\,:\,s_{\mathcal{N}}(e)\subseteq\emptyset\},\\
      [3]=&\,\{e\,:\,s_{\mathcal{N}}(e)\subseteq[2]\},\\
      [6]=&\,\{e\,:\,s_{\mathcal{N}}(e)\subseteq[4]\},\\
    \end{aligned}
  &
    \begin{aligned}[t]
      [7,9]=&\,\{e\,:\,s_{\mathcal{N}}(e)\subseteq[4,5]\},\\
      [4,9]=&\,\{e\,:\,s_{\mathcal{N}}(e)\subseteq[2,5]\},\\
      [9]=&\,\{e\,:\,s_{\mathcal{N}}(e)\subseteq[5]\}.\\
    \end{aligned} 
\end{align*}  
In $R$, the connected cyclic flats are (denoting the support of $e$ by
$s_{\mathcal{A}}(e)$)
\begin{align*}
  &
    \begin{aligned}[t]
      \emptyset=&\,\{e\,:\,s_{\mathcal{A}}(e)\subseteq\emptyset\},\\
      \{1,2,6\}=&\,\{e\,:\,s_{\mathcal{A}}(e)\subseteq[2]\},\\
      \{1,2,3,4,6,7\}=&\,\{e\,:\,s_{\mathcal{A}}(e)\subseteq[4]\},\\        
    \end{aligned}
  &
    \begin{aligned}[t]
      \{4,5,9\}=&\,\{e\,:\,s_{\mathcal{A}}(e)\subseteq[4,5]\},\\
      \{2,3,4,5,8,9\}=&\,\{e\,:\,s_{\mathcal{A}}(e)\subseteq[2,5]\},\\
      [9]=&\,\{e\,:\,s_{\mathcal{A}}(e)\subseteq[5]\}.\\
    \end{aligned} 
\end{align*}  
We see that the connected flats in $M$ and $R$ consisting of the
elements with support in some interval $I$ of $[5]$ have the same size
and rank, and this extends to all cyclic flats.  This illustrates the
next result, that the lattice path and rook matroids coming from the
same lattice path diagram have the same configuration.  This result
strengthens \cite[Theorem 3.38]{rook}, which shows that $M[P,Q]$ and
$R[P,Q]$ have the same Tutte polynomial.  It also proves
\cite[Conjecture 3.39]{rook}: any valuative invariant is the same on
$M[P,Q]$ and $R[P,Q]$.

\begin{thm}\label{thm:sameconfig}
  Fix lattice paths $P$ and $Q$ from $(0,0)$ to $(n-r,r)$ with $P$
  never rising above $Q$.  The lattice path matroid $M=M[P,Q]$ and the
  rook matroid $R=R[P,Q]$ have the same configuration.
\end{thm}

\begin{proof}
  By the observations about direct sums above, it suffices to prove
  this theorem when $M$ and $R$ are connected, that is, $P$ and $Q$
  intersect only at $(0,0)$ and $(n-r,r)$, so we make that assumption.
  Let $\mathcal{N}$ be the path presentation $(N_1,N_2,\ldots,N_r)$ of
  $M$.  Let $\mathcal{A}$ be the presentation $(A_1,A_2,\ldots,A_r)$
  that we used to define $R$.  For an interval $I$ in $[r]$, let
  $$S^M_I=\{e\in[n]\,:\,s_{\mathcal{N}}(e)\subseteq I\} \quad \text{
    and } \quad S^R_I=\{e\in[n]\,:\,s_{\mathcal{A}}(e)\subseteq I\}.$$
  It is easy to see that $S^M_I$ is a flat of $M$, as is $S^R_I$ for
  $R$.  Also, $S^M_\emptyset=S^R_\emptyset=\emptyset$, which is a
  cyclic flat of both $M$ and $R$.  By Lemma \ref{lem:suppint}, each
  connected flat of $M$ is $S^M_I$ for some interval $I$ in $[r]$, and
  likewise for $R$.  We have $I\subseteq S^R_I$, so
  $S^R_I\ne \emptyset$ when $I\ne \emptyset$. The key to the proof is
  establishing the following claim: for any nonempty interval $I$ in
  $[r]$, the set $S^M_I$ is a connected flat of $M$ with
  $|S^M_I|\geq 2$ if and only if $S^R_I$ is a connected flat of $R$
  with $|S^R_I|\geq 2$, and in that case, $|S^M_I|=|S^R_I|$ and
  $r_M(S^M_I)=r_R(S^R_I)$.

  We label each north step in the lattice path diagram with the
  position it has in each path that contains it, as illustrated in
  Figure \ref{fig:runningLPMex}.  Consider a nonempty interval
  $I=[s,t]$ in $[r]$.  First assume that $S^M_I=\emptyset$.  We claim
  that the flat $S^R_I$ of $R$ is not cyclic.  Having
  $S^M_I=\emptyset$ implies that any label on a north step in a row
  between rows $s$ and $t$ also labels a north step in either row
  $s-1$ or row $t+1$.  From that, it follows that each column extends
  either below row $s$ or above row $t$, and so $S^R_I=I$, which, as
  needed, is either a singleton or disconnected.  Now assume that
  $S^M_I\ne\emptyset$.  Since the sets in $\mathcal{N}$ are intervals
  $[l_i,u_i]$ with $l_1<l_2<\cdots<l_r$ and $u_1<u_2<\cdots<u_r$, it
  follows that $S^M_I$ is an interval, say $[a,b]$, in $[n]$.  In the
  lattice path diagram, row $s$ is the first row in which some north
  step has label $a$.  Note that $s=1$ if and only if $a=1$ by our
  assumption about $P$ and $Q$.  Consider the case with $s>1$, and so
  $a>1$.  Row $s-1$ has a north step labeled $a-1$ (since
  $a-1\not\in S^M_I$), and none labeled $a$, so the north step labeled
  $a-1$ in row $s-1$ is in $P$.  If the north step labeled $a$ in row
  $s$ is also in $P$, then $S^M_I\cap N_s=\{a\}$, so $a$ is a coloop
  of $M|S^M_I$, and so either $|S^M_I|=1$ or $M|S^M_I$ is
  disconnected.  In that case, $s$ is the only element of $S^R_I$ that
  has $s$ in its support, so $R|S^R_I$ has $s$ as a coloop, and so
  either $|S^R_I|=1$ or $R|S^R_I$ is disconnected.  Thus, we may focus
  on the case in which the north step labeled $a$ in row $s$ is just
  above a north-east corner of $P$.  By symmetry, when $t\ne r$
  (equivalently, $b\ne n$), the matroids $M|S^M_I$ and $R|S^R_I$ are
  disconnected or have singleton ground sets unless the north step
  labeled $b$ in row $t$ has an east-north corner of $Q$ right above
  it.

  Let $(a_1,a_2)$ be the lowest point on the north step labeled $a$ in
  row $s$, which is $(0,0)$ if $a=s=1$, and otherwise is a north-east
  corner of $P$.  Let $(b_1,b_2)$ be the highest point on the north
  step labeled $b$ in row $t$, which is $(n-r,r)$ if $b=n$ and $t=r$,
  and otherwise is an east-north corner of $Q$.  If $b_1\leq a_1$,
  then $M|S^M_I$ is the free matroid on $[a,b]$ and $R|S^R_I$ is the
  free matroid on $I$; thus, each is connected if and only if
  $|I|=1$. (See Figure \ref{fig:corners}.)  Now assume that
  $a_1<b_1$. To get $M|S^M_I$ and $R|S^R_I$, restrict the diagram to
  the region between these corners and take the resulting lattice path
  or rook matroid.  Both $M|S^M_I$ and $R|S^R_I$ are connected since
  the bounding paths have no common internal points, and they have the
  same rank, namely, $|I|$, and the same number of elements (the
  number of rows plus the number of columns in the restricted
  diagram).  This completes the proof of the claim.

  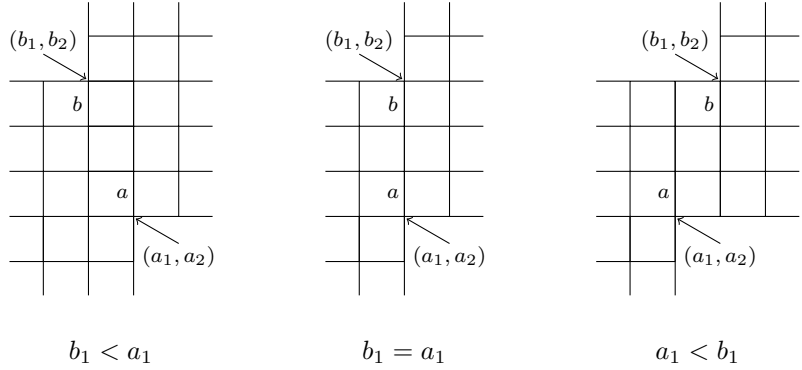
\begin{figure}
    \centering
    \begin{tikzpicture}[scale=0.6]      
      \draw (-7,0) grid (-5,4); %
      \draw (-6,1) grid (-4,5); %

      \draw (-7.75,0) -- (-7,0); %
      \draw (-7.75,1) -- (-7,1); %
      \draw (-7.75,2) -- (-7,2); %
      \draw (-7.75,3) -- (-7,3); %
      \draw (-7.75,4) -- (-7,4); %
      \draw (-3.25,1) -- (-4,1); %
      \draw (-3.25,2) -- (-4,2); %
      \draw (-3.25,3) -- (-4,3); %
      \draw (-3.25,4) -- (-4,4); %
      \draw (-3.25,5) -- (-4,5); %
      \draw (-7,-0.75) -- (-7,1); %
      \draw (-5,-0.75) -- (-5,1); %
      \draw (-6,-0.75) -- (-6,1); %
      \draw (-4,5.75) -- (-4,5); %
      \draw (-5,5.75) -- (-5,5); %
      \draw (-6,5.75) -- (-6,5); %

      \draw (-7,4.9) node {\footnotesize$(b_1,b_2)$};%
      \draw (-4,0.1) node {\footnotesize$(a_1,a_2)$};%

      \draw[->] (-7,4.6) -- (-6.05,4.05); %
      \draw[->] (-4,0.4) -- (-4.95,0.95); %
    
      \draw (-6.25,3.5) node {\footnotesize$b$};%
      \draw (-5.25,1.5) node {\footnotesize$a$};%

      \node at (-5.5,-2) {$b_1<a_1$};%

      %%%%%%%%%%
      
      \draw (0,0) grid (1,4); %
      \draw (1,1) grid (2,5); %
      \draw (-0.75,0) -- (0,0); %
      \draw (-0.75,1) -- (0,1); %
      \draw (-0.75,2) -- (0,2); %
      \draw (-0.75,3) -- (0,3); %
      \draw (-0.75,4) -- (0,4); %
      \draw (2.75,1) -- (2,1); %
      \draw (2.75,2) -- (2,2); %
      \draw (2.75,3) -- (2,3); %
      \draw (2.75,4) -- (2,4); %
      \draw (2.75,5) -- (2,5); %
      \draw (0,-0.75) -- (0,1); %
      \draw (1,-0.75) -- (1,1); %
      \draw (1,5.75) -- (1,5); %
      \draw (2,5.75) -- (2,5); %

      \draw (0,4.9) node {\footnotesize$(b_1,b_2)$};%
      \draw (2,0.1) node {\footnotesize$(a_1,a_2)$};%

      \draw[->] (0,4.6) -- (0.95,4.05); %
      \draw[->] (2,0.4) -- (1.05,0.95); %
    
      \draw (0.75,3.5) node {\footnotesize$b$};%
      \draw (0.75,1.5) node {\footnotesize$a$};%

      \node at (1,-2) {$b_1=a_1$};%

%%%%%%%%%%%%%%

      \draw (6,0) grid (7,4); %
      \draw (7,1) grid (8,4); %
      \draw (8,1) grid (9,5); %

      \draw (5.25,0) -- (6,0); %
      \draw (5.25,1) -- (6,1); %
      \draw (5.25,2) -- (6,2); %
      \draw (5.25,3) -- (6,3); %
      \draw (5.25,4) -- (6,4); %

      \draw (9.75,5) -- (9,5); %
      \draw (9.75,4) -- (9,4); %
      \draw (9.75,3) -- (9,3); %
      \draw (9.75,2) -- (9,2); %
      \draw (9.75,1) -- (9,1); %
      \draw (6,-0.75) -- (6,1); %
      \draw (7,-0.75) -- (7,1); %
      \draw (9,5.75) -- (9,5); %
      \draw (8,5.75) -- (8,5); %
      
      \draw (7.75,3.5) node {\footnotesize$b$};%
      \draw (6.75,1.5) node {\footnotesize$a$};%

      \draw (7,4.9) node {\footnotesize$(b_1,b_2)$};%
      \draw (8,0.1) node {\footnotesize$(a_1,a_2)$};%

      \draw[->] (7,4.6) -- (7.95,4.05); %
      \draw[->] (8,0.4) -- (7.05,0.95); %
      
      \node at (7.5,-2) {$a_1<b_1$};%
    \end{tikzpicture}
    \caption{The options for the corners at $(a_1,a_2)$ and
      $(b_1,b_2)$ in the proof of Theorem \ref{thm:sameconfig}.}
    \label{fig:corners}
  \end{figure}
  
  The connected components of the restriction to a cyclic flat of $M$
  are connected flats of $M$ with at least two elements, and likewise
  for $R$.  Thus, the cyclic flats of $M$ have the form
  $S^M_{I_1}\cup S^M_{I_2}\cup \cdots\cup S^M_{I_k}$ where
  $I_1,I_2,\ldots,I_k$ are pairwise disjoint intervals in $[r]$, and
  likewise for $R$.  If $I_1\cup I_2\cup\cdots\cup I_k$ is an interval
  in $[r]$, then $S^M_{I_1}\cup S^M_{I_2}\cup \cdots\cup S^M_{I_k}$
  could span a connected flat of $M$, but in that case, by what we
  proved above, $S^R_{I_1}\cup S^R_{I_2}\cup \cdots\cup S^R_{I_k}$
  would also span a connected flat of $R$ of the same size and rank.
  When $S^M_{I_1}\cup S^M_{I_2}\cup \cdots\cup S^M_{I_k}$ is a flat of
  $M$, its rank is the sum of the ranks of the components, and
  likewise for the size, and likewise for the counterpart in $R$.
  Thus, $M$ and $R$ have the same configuration.
\end{proof}

By the corollary below, the hypothesis of Theorem
\ref{thm:lpmonuniqueconfig} is equivalent to the lattice path matroid
not being a fundamental transversal matroid.  The proof will use the
observation that all pairs of cyclic flats in a fundamental
transversal matroid $M$ are modular.  To see that, let
$\mathcal{A}=(A_1,A_2,\ldots,A_r)$ be a presentation of $M$ and let
$b_1,b_2,\ldots,b_r$ be elements for which
$s_{\mathcal{A}}(b_i)=\{i\}$ for each $i\in[r]$.  Then
$B=\{b_1,b_2,\ldots,b_r\}$ is a basis of $M$ and for any
$X,Y\in\mathcal{Z}(M)$, we have $r(X)=|X\cap B|$, and likewise for
$Y$, $X\cup Y$, and $X\cap Y$.  (That result extends to arbitrary
collections of cyclic flats in fundamental transversal matroids; see
the proof of \cite[Theorem 3.2]{chartrans}.)

\begin{cor}\label{cor:rookfund}
  For a lattice path matroid $M$, the following statements are
  equivalent:
  \begin{itemize}
  \item[(1)] $M$ is a fundamental transversal matroid,
  \item[(2)] each pair of cyclic flats of $M$ is modular, and
  \item[(3)] no restriction of $M$ to any of its connected components
    has mixed pairs.
  \end{itemize}
\end{cor}

\begin{proof}
  We justified that (1) implies (2) above, and (2) clearly implies
  (3).  To prove that (3) implies (1), note that if $M$ has no mixed
  pairs, then the rook matroid $N$ defined using the same diagram is
  fundamental and has the same configuration as $M$.  By Theorem
  \ref{thm:nomixingconfigunique}, the matroids $M$ and $N$ are
  isomorphic, so $M$ is fundamental.
\end{proof}

While pairs of cyclic flats in fundamental transversal matroids are
modular, that also holds in some transversal matroids that are not
fundamental.  One example is the prism, which is the transversal
matroid on $[6]$ that has the presentation
$([6], \{1,2\}, \{3,4\},\{5,6\})$.

Alexandersson and Jal \cite[Theorem 3.23]{rook} showed that if the
lattice path matroid $M[P,Q]$ has no mixed pairs, then it is
isomorphic to the rook matroid $R[P,Q]$.  The proof that (3) implies
(1) in Corollary \ref{cor:rookfund} shows that this follows from
Theorems \ref{thm:nomixingconfigunique} and \ref{thm:sameconfig}.

\section{Enumeration of non-mixed diagrams}
\label{sec:counting}

In this section we show that exactly $(3^{n-2}+1)/2$ connected lattice
path matroids on $[n]$ have no mixed pairs.  We also refine this count
by rank and note connections with other enumeration problems.  A
corollary of this work, along with enumerative results in \cite{lpm}
and the results in the previous section, is that, asymptotically,
almost no lattice path matroids are configuration unique.

A \emph{diagram} is the shape formed by the unit squares that lie
between two lattice paths $P$ and $Q$ that have the same endpoints,
and with $P$ strictly below $Q$ except at the endpoints.  Thus,
diagrams correspond to connected lattice path matroids on $[n]$ with
the usual order.  A diagram is \emph{non-mixed} if, in the
corresponding lattice path matroid, no pair of initial and final
connected flats is mixed. The \emph{size} of a diagram is the length
of either bounding path minus one, which is the size of the
corresponding lattice path matroid minus one.  We use the following
recursive description of non-mixed diagrams from~\cite[Proposition
3.14]{rook}.

\begin{thm}\label{thm:314rook}
  Non-mixed diagrams are the ones that can be built from a single
  square by any sequence of the following operations: (R) duplicate
  the topmost row; (C) duplicate the rightmost column; (S) add a
  square to the right of the topmost row; (T) add a square above the
  rightmost column. Moreover, the size of the diagram is the number of
  operations performed plus one.
\end{thm}

Thus, each non-mixed diagram of size $m$ can be encoded by at least
one word of length $m-1$ in the alphabet $\{C, R, S, T\}$ ($C$ stands
for \emph{column}, $R$ for \emph{row}, $S$ for \emph{side}, and $T$ for
\emph{top}). Let $D(w)$ denote the non-mixed diagram obtained from the
word $w$.  Several words can yield the same diagram.  For example,
$SRC$, $CRC$, $TCC$, $RCC$, $SSR$, $SCR$, $CSR$, and $CCR$ all give
the same diagram.  The following lemma refines
Theorem~\ref{thm:314rook} by identifying the instances of $S$ and $T$
that can be replaced by other letters and the ones that cannot.
Recall that by a corner at position $(a_1,a_2)$ we mean an $EN$ corner
in the upper path or a $NE$ corner in the lower path, where
$(a_1,a_2)$ are the coordinates of the point where the steps of the
corner meet.

\begin{lemma}\label{lem:corners}
  Let $w$ be any word in the alphabet $\{C, R, S, T\}$. If $D(w)$ has
  a corner at $(a_1,a_2)$, then (a) $w_{a_1+a_2}=S$ if the corner is
  in the lower path, and (b) $w_{a_1+a_2}=T$ if the corner is in the
  upper path. Moreover, if all other appearances of $S$ and $T$ in $w$
  are replaced by $C$ and $R$, respectively, then the resulting word
  gives the same diagram $D(w)$.
\end{lemma}

\begin{proof}
  The first assertion holds since $C$ and $R$ never create corners,
  and $S$ does not create an upper corner and $T$ does not create a
  lower corner.  For the second part, suppose that $w_i=S$ but that
  this $S$ does not create a corner. Let $\overline{w}$ be the word
  $w_1\ldots w_{i-1}$. Note that the rightmost column of the diagram
  $D(\overline{w})$ must have height one, so
  $D(\overline{w}S)=D(\overline{w}C)$.  Similarly, any $T$ that does
  not create a corner can be replaced by $R$.
\end{proof}

It is straightforward to check that replacing any instance of $RC$ by
$CR$ yields the same diagram. This observation and
Lemma~\ref{lem:corners} allow us to associate a unique word to each
non-mixed diagram.  For instance, of the eight words that give the
same diagram in the example above, only $CCR$ satisfies the conditions
in the following result.

\begin{cor}\label{cor:uniqueword}
  Each non-mixed diagram arises from exactly one word in the alphabet
  $\{C, R, S, T\}$ in which $S$ occurs precisely to create a corner in
  the lower path, $T$ occurs precisely to create a corner in the upper
  path, and $RC$ does not occur as a subword (i.e., a word that occurs
  as consecutive letters in the word). In particular, this word does
  not have $SS$ or $TT$ as subwords.
\end{cor}

We introduce a subclass of diagrams that will facilitate the
enumeration of non-mixed diagrams. We say that a diagram is
\emph{thick} if it is non-mixed and all horizontal or vertical
segments joining a point of $P$ with a point of $Q$ through the
interior of the diagram have length at least two.  A thick diagram has
size at least three. An abitrary non-mixed diagram can be decomposed
as a sequence of thick diagrams and single squares as follows: look at
all the rows or columns that meet the next row or column in just one
square, and cut the diagram along the edges that separate these rows
or columns (see Figure~\ref{fig:thick}).

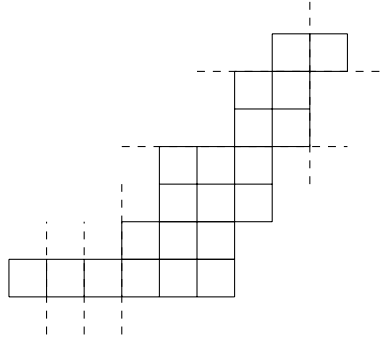
\begin{figure}%[h]
  \centering
  \begin{tikzpicture}[scale=0.5]
    \draw (0,0) grid (6,1); %
    \draw (3,1) grid (6,2);%
    \draw (4,2) grid (7,4);% 
    \draw (6,4) grid (8,6);%
    \draw (7,6) grid (9,7);%

    \draw[dashed] (1,-1) -- (1, 2);%
    \draw[dashed] (2,-1) -- (2, 2);%
    \draw[dashed] (3,-1) -- (3,3);%
    \draw[dashed] (3,4) -- (9,4); %
    \draw[dashed] (5,6) -- (10, 6);%
    \draw[dashed] (8,3) -- (8, 8);%    
  \end{tikzpicture}
  \caption{The decomposition of a non-mixed diagram into two thick
    diagrams and five single squares. This diagram is generated by the
    word $CCCTCTCSRTSRTS$, and the first thick piece alone is
    generated by $CRTCSR$.}
  \label{fig:thick}
\end{figure}

We will use this decomposition of non-mixed diagrams together with the
structure of the words describing them to obtain the generating
functions and exact counting formulas for the number of non-mixed and
thick diagrams.  We will use the symbolic method (see Chapter 1
of~\cite{flajsedg} for a thorough introduction). The main idea is that
if $A(z)=\sum_{n\geq 0} a_n z^n$ is the generating function for the
objects in a collection $\mathcal{A}$ according to some size function
(so that $a_n$ is the number of elements of $\mathcal{A}$ that have
size $n$), then the generating function for finite sequences of
objects of $\mathcal{A}$ is given by $1/(1-A(z))$, where the size of a
sequence is the sum of the sizes of its components.

The following theorem gives the numbers of thick and arbitrary
non-mixed diagrams of fixed size, which turn out to be simple
combinatorial expressions.  Recall that the \emph{Pell numbers} $P_n$
are given by the recurrence $P_{n+2}-2P_{n+1}-P_n=0$ for $n\geq 0$
with the initial conditions $P_0=0$, $P_1=1$ (see sequence
\href{https://oeis.org/A000129}{A000129} in the
OEIS~\cite{OEIS}). Their generating function is $P(z)=z/(1-2z-z^2)$.

\begin{thm}\label{thm:enumeration}
  For $m\geq 3$, there are $P_{m-2}$ thick diagrams of size $m$.  For
  $m\geq 1$, there are $(3^{m-1}+1)/2$ non-mixed diagrams of size $m$.
  Thus, for $n\geq 2$, there are $(3^{n-2}+1)/2$ connected lattice
  path matroids on $[n]$ in which every pair that consists of an
  initial and a final connected flat is non-mixed.
\end{thm}

\begin{proof}
  By Corollary \ref{cor:uniqueword}, we can associate a unique word
  $w=w_1\ldots w_{m-1}$ in the alphabet $\{C, R, S, T\}$ to each
  non-mixed diagram of size $m$.  For a letter $A$, we let
  $A^{\geq 1}$ denote a string of $A$s of length at least one. We use
  $A^{\geq 0}$ analogously.

  We first consider thick diagrams.  A diagram $D(w)$ with no corners
  is thick if and only if $w$ has at least one $C$ and at least one
  $R$, i.e., it has the form $C^{\geq 1}R^{\geq 1}$.  Now suppose that
  the diagram $D(w)$ has at least one corner, and let
  $i_1<i_2<\cdots< i_c$ be the positions in $w$ that correspond to the
  corners of $D(w)$.  The diagram $D(w)$ is thick if and only if
    \begin{enumerate}
    \item the initial subword $w_1\ldots w_{i_1-1}$ has the form
      $C^{\geq 1}R^{\geq 1}$;
    \item for all $j$ with $1\leq j \leq c$, if $w_{i_j}=S$, then the
      subword $w_{i_j+1}\ldots w_{i_{j+1}-1}$ has the form
      $C^{\geq 0}R^{\geq 1}$, and if $w_{i_j}=T$, then the subword
      $w_{i_j+1}\ldots w_{i_{j+1}-1}$ has the form
      $C^{\geq 1}R^{\geq 0}$ (where $w_{i_{c+1}-1}=w_{m-1}$ if $j=c$).
  \end{enumerate}
   
  Let $\mathrm{Th}(z)=\sum_{m\geq 3} t_m z^m$ be the generating
  function for thick diagrams according to size.  From the description
  above, the word $w$ can be split into subwords $s_0,\ldots, s_c$,
  with $c\geq 0$, such that $s_0$ is of the form
  $C^{\geq 1}R^{\geq 1}$ and, when $c>0$, for each $i\in[c]$, the
  subword $s_i$ has two possible forms, $SC^{\geq 0} R^{\geq 1}$ or
  $TC^{\geq 1}R^{\geq 0}$.  Translating this description into
  generating functions gives
  \begin{equation}\label{eq:th}
    \mathrm{Th}(z) = z  \frac{z^2}{(1-z)^2}
    \frac{1}{1-2\frac{z^2}{(1-z)^2}} 
    =\frac{z^3}{1-2z-z^2}=z^2P(z),    
  \end{equation}
  where the first $z$ accounts for the initial square of the diagram.

  We construct an arbitrary non-mixed diagram by gluing consecutive
  terms in a sequence of thick diagrams and single squares, as
  explained above.  Given consecutive terms $D_1$ and $D_2$ (thick
  diagrams or single squares), there are two ways to glue them along
  one edge: either the last north step of the bottom path of $D_1$ is
  glued to the first north step of the top path of $D_2$, or the last
  east step of the top path of $D_1$ is glued to the first east step
  of the bottom path of $D_2$.  Thus, a non-mixed diagram is a
  sequence of diagrams $D_1,\ldots, D_k$ such that $k\geq 1$, each
  $D_i$ is either thick or a single square for all $i \in [k]$, and
  each $D_i$ has a mark on the last step of the top or of the bottom
  path, for all $i \in [k-1]$. This decomposition gives the following
  generating function for the number of non-mixed diagrams according
  to size:
  $$\frac{1}{1-(2\mathrm{Th}(z)+2z)}(\mathrm{Th}(z)+z)=
  \frac{z-2z^2}{(1-z)(1-3z)}=\sum_{m\geq 1} \frac{1}{2}(3^{m-1}+1)
  z^m.\qedhere$$ 
\end{proof}
  
As shown in \cite[Section 4]{lpm}, the number of connected lattice
path matroids on $[n]$, with the usual order, is the Catalan number
$C_{n-1}$.  (As above, that does not take isomorphism into account;
that counts diagrams.  The order of magnitude is the same if we count
up to matroid isomorphism; see \cite[Theorem 4.2]{lpm}.)  It is well
known that the Catalan numbers $C_n$ grow like
$4^n/(n^{3/2}\sqrt{\pi})$, so $\lim_{n\to\infty}3^{n-2}/C_{n-1}=0$,
which gives the  corollary below.

\begin{cor}
  Asymptotically, almost no lattice path matroids are configuration
  unique.
\end{cor}

The decompositions in the proof of Theorem~\ref{thm:enumeration} allow
us to refine the enumeration by taking the ranks of the corresponding
lattice path matroids into account.  The rank of the matroid is the
number of rows in the corresponding diagram.  In this enumeration we
also encounter some known combinatorial numbers, which we review next.

Recall that the \emph{Delannoy numbers} $d_{i,j}$ count the number of
paths from $(0,0)$ to $(i,j)$ with steps $(1,0), (0,1)$ and $(1,1)$
(see~\cite{delannoy} and sequence
\href{https://oeis.org/A008288}{A008288} in the
OEIS~\cite{OEIS}). Their bivariate generating function is
$$\mathrm{Del}(z,y)=\sum_{i,j\geq 0} d_{i,j}z^iy^j=\frac{1}{1-(z+y+zy)}.$$
A partition $\{S_1,\ldots, S_r\}$ of the set $[m]$ is
\emph{order-consecutive} if there is some permutation
$k_1,\ldots, k_r$ of $[r]$ such that the sets
$S_{k_1}\cup \cdots \cup S_{k_{\ell}}$ are intervals for all
$ \ell \in [r]$ (see~\cite{orderconsec} and sequence
\href{https://oeis.org/A056241}{A056241} in the OEIS~\cite{OEIS}). The
bivariate generating function for the number $\mathrm{oc}_{m,r}$ of
order-consecutive partitions of $[m]$ with $r$ parts is
$$\mathrm{OC}(z,y)=\sum_{m\geq r\geq 1} \mathrm{oc}_{m,r} z^m
y^r=zy\frac{1-z(1+y)}{1-2z(1+y)+z^2(1+y+y^2)}.$$

\begin{thm}\label{thm:bivariate}
  The number of thick diagrams of size $m$ with $r$ rows is the
  Delannoy number $d_{m-r-1, r-2}$. The number of non-mixed diagrams
  of size $m$ with $r$ rows is the number of $r$-part order-consecutive
  partitions of $[m]$.
\end{thm}

\begin{proof}
  Define the bivariate generating function
  $\mathrm{Th}(z,y)=\sum_{m,r} t_{m,r} y^r z^m$, where $t_{m,r}$ is
  the number of thick diagrams of size $m$ with $r$ rows. Note that
  $t_{m,r}>0$ if and only if $m>r>1$.

  For thick diagrams, the rank increases by one exactly when an
  operation $R$ or $T$ is performed. With this observation, the
  analogue of equation~(\ref{eq:th}) is
      $$\mathrm{Th}(z,y)
    = zy \frac{z^2y}{(1-z)(1-zy)} \frac{1}{1-
      \frac{2z^2y}{(1-z)(1-zy)}}
    =\frac{z^3y^2}{1-z-zy-z^2y}=z^3y^2\mathrm{Del}(z,zy).$$
    The
  coefficient of $z^iy^j$ in $\mathrm{Del}(z,zy)$ is $d_{i-j,j}$, so
  the coefficient of $z^m y^r$ in $\mathrm{Th}(z,y)$ is
  $d_{m-r-1, r-2}$.
    
  We obtain a general non-mixed diagram from a sequence of thick
  diagrams or single squares, with two ways to glue consecutive
  diagrams.  If we glue two diagrams along east steps, then the rank
  of the new diagram is the sum of the ranks of the two original
  diagrams; if we glue along north steps, then we must substract one
  from this sum. This yields the following generating function
  $$\frac{1}{1-(\mathrm{Th}(z,y)+zy+
    \frac{\mathrm{Th}(z,y)}{y}+z)}(\mathrm{Th}(z,y)+zy)
  =zy\frac{1-z(1+y)}{1-2z(1+y)+z^2(1+y+y^2)}.$$  
\end{proof}

It would be interesting to find bijective proofs of
Theorems~\ref{thm:enumeration} and~\ref{thm:bivariate}.

\section{Configuration-unique matroids}\label{sec:confrecon}

To start, we give infinitely many pairs $M,M'$ of matroids that have
the same $\mathcal{G}$-invariant but different configurations, and $M$
is a lattice path matroid that is not configuration unique while $M'$
is configuration unique.  Fix positive integers $b$ and $k$.  Let $M$
be $M[P,Q]$ where $P$ is $E^{2b+k}N^kE^bN^kE^bN^k$ and $Q$ is
$N^kE^bN^kE^bN^kE^{2b+k}$.  Figure \ref{fig:diffconfig} illustrates
this for $b=k=1$.  Figure \ref{fig:diffconfigN} shows a matroid $M'$
that has the same $\mathcal{G}$-invariant as that in Figure
\ref{fig:diffconfig}, but the configurations differ.  Figure
\ref{fig:diffconfig2} shows the lattice of cyclic flats of $M$ in the
general case, as well as the lattice of cyclic flats of a matroid $M'$
that generalizes that in Figure \ref{fig:diffconfigN}, that is not a
lattice path matroid, and that has the same $\mathcal{G}$-invariant as
$M$.  It is easy to verify properties (Z0)--(Z3) in Theorem
\ref{thm:axioms} for the sets and ranks given in Figure
\ref{fig:diffconfig2}.  The fact that the configurations yield the
same $\mathcal{G}$-invariant is an instance of \cite[ Theorem
4.1]{sameG}.

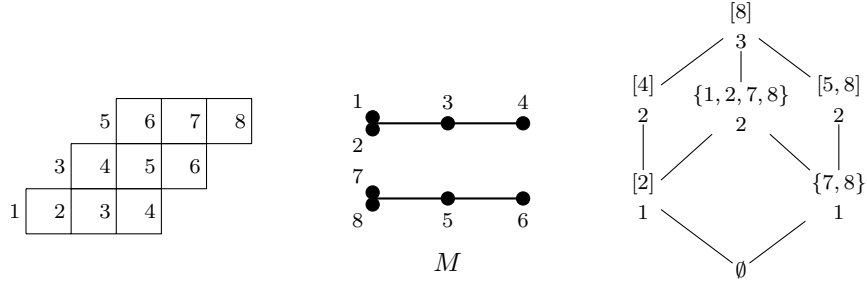
\begin{figure}
  \centering
  \begin{tikzpicture}[scale=0.6]
    \draw (0,0) grid (3,1); %
    \draw (1,1) grid (4,2); %
    \draw (2,2) grid (5,3);%

    \draw (-0.25,0.5) node {\footnotesize$1$};%
    \draw (0.75,0.5) node {\footnotesize$2$};%
    \draw (1.75,0.5) node {\footnotesize$3$};%
    \draw (2.75,0.5) node {\footnotesize$4$};%   

    \draw (0.75,1.5) node {\footnotesize$3$};%
    \draw (1.75,1.5) node {\footnotesize$4$};%
    \draw (2.75,1.5) node {\footnotesize$5$};%
    \draw (3.75,1.5) node {\footnotesize$6$};%

    \draw (1.75,2.5) node {\footnotesize$5$};%
    \draw (2.75,2.5) node {\footnotesize$6$};%
    \draw (3.75,2.5) node {\footnotesize$7$};%
    \draw (4.75,2.5) node {\footnotesize$8$};%

        \node at (1,-0.8) {\color{white}.};%
  \end{tikzpicture}
    \hspace{1cm}
  \begin{tikzpicture}[scale=1]
    \filldraw (0,4.92) node[below left] {\footnotesize $2$} circle
    (2.5pt);%
    \filldraw (0,5.08) node[above left] {\footnotesize $1$} circle
    (2.5pt);%
    \filldraw (1,5) node[above=2pt] {\footnotesize $3$} circle
    (2.5pt);%
    \filldraw (2,5) node[above=2pt] {\footnotesize $4$} circle
    (2.5pt);%
    \filldraw (0,3.92) node[below left] {\footnotesize $8$} circle
    (2.5pt);%
    \filldraw (0,4.08) node[above left] {\footnotesize $7$} circle
    (2.5pt);%
    \filldraw (1,4) node[below=2pt] {\footnotesize $5$} circle
    (2.5pt);%
    \filldraw (2,4) node[below=2pt] {\footnotesize $6$} circle
    (2.5pt);%
    \draw[thick](0,4)--(2,4);%
    \draw[thick](0,5)--(2,5);%
    \node at (1,3.15) {$M$};%
  \end{tikzpicture}
    \hspace{1cm}
  \begin{tikzpicture}[scale=1.3]
    \node[inner sep = 0.3mm] (em) at (0,-0.25) {\footnotesize
      $\emptyset$};%

    \node[inner sep = 0.3mm] (1a) at (-1,0.5) {\footnotesize
      $\displaystyle\genfrac{}{}{0pt}{}{[2]}{1}$};%
    \node[inner sep = 0.3mm] (1b) at (1,0.5) {\footnotesize
      $\displaystyle\genfrac{}{}{0pt}{}{\{7,8\}}{1}$};%
    
    \node[inner sep = 0.3mm] (2a) at (-1,1.5) {\footnotesize
      $\displaystyle\genfrac{}{}{0pt}{}{[4]}{2}$};%
    \node[inner sep = 0.3mm] (2b) at (0,1.4) {\footnotesize
      $\displaystyle\genfrac{}{}{0pt}{}{\{1,2,7,8\}}{2}$};%
    \node[inner sep = 0.3mm] (2c) at (1,1.5) {\footnotesize
      $\displaystyle\genfrac{}{}{0pt}{}{[5,8]}{2}$};%
   
    \node[inner sep = 0.3mm] (3) at (0,2.25) {\footnotesize
      $\displaystyle\genfrac{}{}{0pt}{}{[8]}{3}$};%

    \foreach \from/\to in {em/1a,em/1b,1a/2a,1a/2b,1b/2b,1b/2c,2a/3,
      2b/3,2c/3} \draw(\from)--(\to);%
  \end{tikzpicture}
  \caption{A lattice path matroid $M$ and its lattice of cyclic flats.}
  \label{fig:diffconfig}
\end{figure}

\begin{figure}
  \centering
   \begin{tikzpicture}[scale=1]
     \filldraw (5.75,4.42) node[below left] {\footnotesize $2$} circle
     (2.5pt);%
     \filldraw (5.75,4.58) node[above left] {\footnotesize $1$} circle
     (2.5pt);%
     \filldraw (6.75,4.75) node[above=2pt] {\footnotesize $3$} circle
     (2.5pt);%
     \filldraw (7.75,5) node[above=2pt] {\footnotesize $4$} circle
     (2.5pt);%
     \filldraw (6.75,4.25) node[below=2pt] {\footnotesize $5$} circle
     (2.5pt);%
     \filldraw (7.75,4) node[below=2pt] {\footnotesize $6$} circle
     (2.5pt);%
     \filldraw (8.25,4.42) node[below right] {\footnotesize $8$}
     circle (2.5pt);%
     \filldraw (8.25,4.58) node[above right] {\footnotesize $7$}
     circle (2.5pt);%
     \draw[thick](5.75,4.5)--(7.75,4);%
     \draw[thick](5.75,4.5)--(7.75,5);%
     \node at (6.75,3) {$M'$};%
   \end{tikzpicture}
   \hspace{1.5cm}
   \begin{tikzpicture}[scale=1.3]
    \node[inner sep = 0.3mm] (em) at (1,-0.25) {\footnotesize
      $\emptyset$};%

    \node[inner sep = 0.3mm] (1a) at (0,0.5)  {\footnotesize
      $\displaystyle\genfrac{}{}{0pt}{}{\{7,8\}}{1}$};%
    \node[inner sep = 0.3mm] (1b) at (2,0.5)  {\footnotesize
      $\displaystyle\genfrac{}{}{0pt}{}{[2]}{1}$};%

    \node[inner sep = 0.3mm] (2a) at (1,1.4)  {\footnotesize
      $\displaystyle\genfrac{}{}{0pt}{}{\{1,2,7,8\}}{2}$};%
    \node[inner sep = 0.3mm] (2b) at (2,1.4)  {\footnotesize
      $\displaystyle\genfrac{}{}{0pt}{}{[4]}{2}$};%
    \node[inner sep = 0.3mm] (2c) at (3,1.4)  {\footnotesize
      $\displaystyle\genfrac{}{}{0pt}{}{\{1,2,5,6\}}{2}$};%
 
    \node[inner sep = 0.3mm] (3) at (2,2.25)  {\footnotesize
      $\displaystyle\genfrac{}{}{0pt}{}{[8]}{3}$};%

    \foreach \from/\to in {em/1a,em/1b,1a/2a,1b/2a,1b/2b,1b/2c,2a/3,
      2b/3,2c/3} \draw(\from)--(\to);%
  \end{tikzpicture}
  \caption{A matroid $M'$ and its lattice of cyclic flats.  This
    matroid has the same $\mathcal{G}$-invariant as that in Figure
    \ref{fig:diffconfig}, but the configurations differ.}
  \label{fig:diffconfigN}
\end{figure}
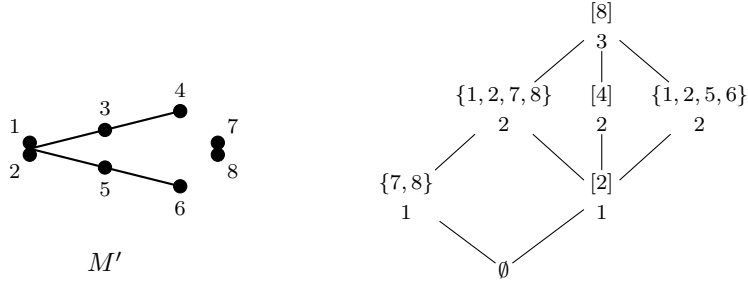

\begin{figure}
  \centering
  \begin{tikzpicture}[scale=1.3]
    \node[inner sep = 0.3mm] (em) at (0,-0.25) {\footnotesize
      $\emptyset$};%

    \node[inner sep = 0.3mm] (1a) at (-1,0.5) {\footnotesize
      $\displaystyle\genfrac{}{}{0pt}{}{W}{k}$};%
    \node[inner sep = 0.3mm] (1b) at (1,0.5) {\footnotesize
      $\displaystyle\genfrac{}{}{0pt}{}{Z}{k}$};%
    
    \node[inner sep = 0.3mm] (2a) at (-1,1.5) {\footnotesize
      $\displaystyle\genfrac{}{}{0pt}{}{W\cup X}{2k}$};%
    \node[inner sep = 0.3mm] (2b) at (0,1.4) {\footnotesize
      $\displaystyle\genfrac{}{}{0pt}{}{W\cup Z}{2k}$};%
    \node[inner sep = 0.3mm] (2c) at (1,1.5) {\footnotesize
      $\displaystyle\genfrac{}{}{0pt}{}{Y\cup Z}{2k}$};%
   
    \node[inner sep = 0.3mm] (3) at (0,2.25) {\footnotesize
      $\displaystyle\genfrac{}{}{0pt}{}{W\cup X\cup Y\cup Z}{3k}$};%

    \foreach \from/\to in {em/1a,em/1b,1a/2a,1a/2b,1b/2b,1b/2c,2a/3,
      2b/3,2c/3} \draw(\from)--(\to);%

    \node at (0,-0.7) {$\mathcal{Z}(M)$};%

  \end{tikzpicture}
    \hspace{1cm}
  \begin{tikzpicture}[scale=1.3]
    \node[inner sep = 0.3mm] (em) at (1,-0.25) {\footnotesize
      $\emptyset$};%

    \node[inner sep = 0.3mm] (1a) at (0,0.5)  {\footnotesize
      $\displaystyle\genfrac{}{}{0pt}{}{Z}{k}$};%
    \node[inner sep = 0.3mm] (1b) at (2,0.5)  {\footnotesize
      $\displaystyle\genfrac{}{}{0pt}{}{W}{k}$};%

    \node[inner sep = 0.3mm] (2a) at (1,1.4)  {\footnotesize
      $\displaystyle\genfrac{}{}{0pt}{}{W\cup Z}{2k}$};%
    \node[inner sep = 0.3mm] (2b) at (2,1.4)  {\footnotesize
      $\displaystyle\genfrac{}{}{0pt}{}{W\cup X}{2k}$};%
    \node[inner sep = 0.3mm] (2c) at (3,1.4)  {\footnotesize
      $\displaystyle\genfrac{}{}{0pt}{}{W\cup Y}{2k}$};%
 
    \node[inner sep = 0.3mm] (3) at (2,2.25)  {\footnotesize
      $\displaystyle\genfrac{}{}{0pt}{}{W\cup X\cup Y\cup Z}{3k}$};%

    \foreach \from/\to in {em/1a,em/1b,1a/2a,1b/2a,1b/2b,1b/2c,2a/3,
      2b/3,2c/3} \draw(\from)--(\to);%

         \node at (1.5,-0.7) {$\mathcal{Z}(M')$};%
  \end{tikzpicture}
  \caption{The lattices of cyclic flats of a lattice path matroid $M$
    and a matroid $M'$ that have different configurations but the same
    $\mathcal{G}$-invariant.  The sets $W$, $X$, $Y$, and $Z$ are the
    intervals $[(i-1)(b+k)+1,i(b+k)]$ for $i\in[4]$, respectively;
    each has $b+k$ elements.}
  \label{fig:diffconfig2}
\end{figure}
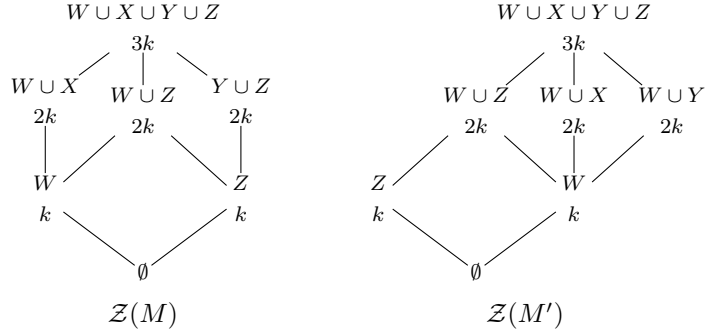

The matroid $M$ in Figure \ref{fig:diffconfig2} is not configuration
unique since $(W\cup X,Y\cup Z)$ is a mixed pair.  However, the next
theorem applies to the matroid $M'$ in Figure \ref{fig:diffconfig2},
which therefore is configuration unique.  So, $M'$ is configuration
unique but not $\mathcal{G}$ unique.
  
\begin{thm}\label{thm:congifrecon}
  For a matroid $M$, if, for all $A,B\in\mathcal{Z}(M)$, the pair
  $(A,B)$ is modular and $A\cap B\in\mathcal{Z}(M)$, then $M$ is
  configuration unique.
\end{thm}

\begin{proof}
  Without loss of generality, we may assume that $M$ has no coloops.
  Let $N$ be a matroid that has the same configuration as $M$, so $N$
  has no coloops and there is a lattice isomorphism
  $\Phi:\mathcal{Z}(M)\to\mathcal{Z}(N)$ that preserves the size and
  rank of each cyclic flat.  To show that $M$ is configuration unique,
  we construct a bijection $\phi:E(M)\to E(N)$ for which
  $\Phi(A)=\phi(A)$ for all $A\in \mathcal{Z}(M)$, which therefore is
  the isomorphism we need.

  We claim that $A'\cap B'\in\mathcal{Z}(N)$ for all
  $A',B'\in\mathcal{Z}(N)$.  Take $A,B\in\mathcal{Z}(M)$ with
  $A'=\Phi(A)$ and $B'=\Phi(B)$.  The assumptions that
  $A\cap B\in \mathcal{Z}(M)$ (so $A\meet B=A\cap B$) and that $(A,B)$
  is a modular pair give
  $$r_M(A)+r_M(B)= r_M(A\join B)+r_M(A\meet B).$$  From that equality,
  the properties of $\Phi$ give
  $$r_N(A')+r_N(B')= r_N(A'\join B')+r_N(A'\meet B').$$ The
  submodular inequality
  $r_N(A')+r_N(B')\geq r_N(A'\cup B')+r_N(A'\cap B')$ along with the
  inclusion $A'\meet B'\subseteq A'\cap B'$ force the flat $A'\cap B'$
  to be the cyclic flat $A'\meet B'$, so, as claimed,
  $A'\cap B'\in\mathcal{Z}(N)$.
  
  For each $A\in \mathcal{Z}(M)$, let its \emph{height in
    $\mathcal{Z}(M)$}, denoted $h(A)$, be the largest integer $h$ for
  which there is a chain
  $A_1\subsetneq A_2\subsetneq \cdots \subsetneq A_h\subsetneq A$ with
  all $A_i$ in $\mathcal{Z}(M)$.  Thus, $\cl_M(\emptyset)$ has height
  $0$ and the covers of $\cl_M(\emptyset)$ in $\mathcal{Z}(M)$ have
  height $1$.  Let $E(M)$ have height $t$, which, adapting the
  definition to $N$, is also the height of $E(N)$.  Note that the
  height of $A$ in $\mathcal{Z}(M)$ is the height of $\Phi(A)$ in
  $\mathcal{Z}(N)$.  For each element $e\in E(M)$, let the
  \emph{height $h(e)$ of $e$} be
  $\min\{h(A)\,:\,e\in A\in\mathcal{Z}(M)\}$, the least height of a
  cyclic flat that contains $e$.  Thus, $h(e)=0$ if and only if $e$ is
  a loop, and elements of height $t$ are in no proper cyclic flats.
  Let $E_i(M)=\{e\in E(M)\,:\,h(e)\leq i\}$ for $i$ with
  $0\leq i\leq t$, and let $E_i(N)$ be defined similarly.  We define
  $\phi$ recursively: as $i$ ranges from $0$ to $t$, we define
  bijections $\phi_i:E_i(M)\to E_i(N)$ for which $\Phi(A)=\phi_i(A)$
  for all $A\in \mathcal{Z}(M)$ with $h_M(A)\leq i$, and, if $i>0$,
  the restriction of $\phi_i$ to $E_{i-1}(M)$ is $\phi_{i-1}$.  Thus,
  the bijection $\phi$ that we want is $\phi_t$.  For $i=0$, since
  $|\cl_M(\emptyset)|=|\cl_N(\emptyset)|$, let $\phi_0$ be any
  bijection from $\cl_M(\emptyset)$ onto $\cl_N(\emptyset)$.  Now
  assume that for some $k\in[t]$, the bijection
  $\phi_{k-1}:E_{k-1}(M)\to E_{k-1}(N)$ has the required properties.
  Each element $e\in E_k(M)-E_{k-1}(M)$ is in exactly one cyclic flat
  $F$ with $h(F)=k$, for if $e$ were in two such cyclic flats, say $F$
  and $F'$, then $e\in F\cap F'$, and, by the hypotheses of the
  theorem, $F\cap F'$ is a cyclic flat, and clearly $h(F\cap F')<k$,
  contrary to having $e \not\in E_{k-1}(M)$.  Let $F$ be a cyclic flat
  with $h(F)=k$.  Now if $f\in F$ and $h(f)<k$, then $f$ is in a
  cyclic flat $F'$ with $h(F')<k$, and so $f$ is in the cyclic flat
  $F\cap F'$, which is properly contained in $F$ and has height less
  than $k$.  Let $F_0$ be the union of the cyclic flats that are
  properly contained in $F$, so $E_{k-1}(M)\cap F=F_0$.  By the
  principle of inclusion/exclusion, $|F_0|$ can be found from the
  sizes of the cyclic flats that are properly contained in $F$, and
  the sizes of intersections of such cyclic flats.  Since (i) such
  intersections are cyclic flats that are properly contained in $F$,
  (ii) $\mathcal{Z}(N)$ is closed under intersections, and (iii) the
  bijection $\Phi$ preserves sizes and inclusions of cyclic flats, we
  get $|E_{k-1}(N)\cap \Phi(F)|=|F_0|$.  Thus,
  $|F-E_{k-1}(M)| =|\Phi(F)-E_{k-1}(N)|$. Therefore we can extend
  $\phi_{k-1}:E_{k-1}(M) \to E_{k-1}(N)$ to $\phi_k:E_k(M) \to E_k(N)$
  by, for any cyclic flat $F$ of height $k$, letting the restriction
  to $F-E_{k-1}(M)$ be any bijection onto $\Phi(F)-E_{k-1}(N)$; such
  an extension is well defined since each element in
  $E_k(M)-E_{k-1}(M)$ is in exactly one cyclic flat of height $k$.  As
  noted above, $\phi_t$ is the isomorphism of $M$ onto $N$ that we
  needed.
\end{proof}

By duality, the equality
$\mathcal{Z}(M^*)=\{E(M)-A\,:\,A\in\mathcal{Z}(M)\}$, and Lemmas
\ref{lem:modprcompdual} and \ref{lem:compdualconfig}, we get the
corollary below.

\begin{cor}\label{cor:congifrecon}
  For a matroid $M$, if, for all $A,B\in\mathcal{Z}(M)$, the pair
  $(A,B)$ is modular and $A\cup B\in\mathcal{Z}(M)$, then $M$ is
  configuration unique.
\end{cor}

Neither Theorem \ref{thm:congifrecon} nor Corollary
\ref{cor:congifrecon} has Theorem \ref{thm:nomixingconfigunique} as a
corollary.

Another sufficient condition for configuration uniqueness is having
all elements of $M$ be in $2$-circuits since that implies that all
flats of $M$ are cyclic, so the configuration is an unlabeled copy of
the lattice of flats, with the size of each parallel class given.
That observation is used in the first of two constructions discussed
below that produce pairs of non-isomorphic matroids that are
configuration unique but have the same $\mathcal{G}$-invariant.

For the first construction, start with any two non-isomorphic matroids
$M$ and $N$ with $\mathcal{G}(M)=\mathcal{G}(N)$ and an integer
$k\geq 1$; for each $e\in E(M)$, let $X_e$ be a set of size $k$ that
is disjoint from $E(M)$ and satisfies $X_e\cap X_f=\emptyset$ when
$e\ne f$, and add the elements of $X_e$ parallel to $e$ to get a
matroid $M^k$; form $N^k$ similarly.  Thus, a flat of $h$ elements in
$M$ gives rise to a flat of $(k+1)h$ elements in $M^k$, and likewise
for $N^k$.  By what we just noted, both $M^k$ and $N^k$ are
configuration unique.  However, the formulation of the
$\mathcal{G}$-invariant using sizes of differences in flags of flats
shows that $\mathcal{G}(M^k)=\mathcal{G}(N^k)$.

The next construction, discussed after Theorem \ref{thm:coneconfigu},
will use the free $m$-cone, where $m$ is a positive integer.  For a
matroid $M$ with no loops, the \emph{free $m$-cone} of $M$, denoted
$Q_m(M)$, is formed by extending $M$ by adding a coloop, say $a$, and,
by taking iterated principal extensions, for each $e\in E(M)$, adding
$m$ points freely to the line spanned by $a$ and $e$.  When $M$ and
$m$ are understood, we shorten $Q_m(M)$ to $Q$. (How many elements are
on a line through $a$ is determined by the size of the corresponding
rank-$1$ flat of $M$; a rank-$1$ flat consisting of $h$ parallel
elements in $M$ gives rise to a line of $Q$ with $h(m+1)+1$ elements.)
That is one of the two views of free $m$-cones from \cite{freecone},
where this construction was introduced; the other approach specifies
the cyclic flats of $Q$ and their ranks.  The cyclic flats of $Q$ are
those of $M$ along with all unions of the form
$\bigcup_{x\in X}\cl_Q(a,x)$ as $X$ ranges over the nonempty flats of
$M$.

\begin{figure}
  \centering
\begin{tikzpicture}[scale=1]

  \filldraw (0,0) circle (2.5pt); %
  \filldraw (1,0) circle (2.5pt); %
  \filldraw (2,0.09) circle (2.5pt); %
  \filldraw (2,-0.09) circle (2.5pt); %
  \filldraw (3,0) circle (2.5pt); %
  \filldraw (0.75,1) circle (2.5pt); %
  \filldraw (1.2,0.8) circle (2.5pt); %
  \filldraw (1.8,0.8) circle (2.5pt); %
  \filldraw (2.25,1) circle (2.5pt); %
  \filldraw (1.7,1.2) circle (2.5pt); %
  \filldraw (1.5,2) circle (2.5pt); %
        
  \draw[thick](0,0)--(3,0);%
  \draw[thick](1.5,2)--(0,0);%
  \draw[thick](1.5,2)--(1,0);%
  \draw[thick](1.5,2)--(2,0);%
  \draw[thick](1.5,2)--(3,0);%

  \node at (1.5,-0.75) {$Q$};%
\end{tikzpicture}
\hspace{1cm}
\begin{tikzpicture}[scale=1]

  \filldraw (0,0) circle (2.5pt); %
  \filldraw (1,0) circle (2.5pt); %
  \filldraw (2,0.09) circle (2.5pt); %
  \filldraw (2,-0.09) circle (2.5pt); %
  \filldraw (3,0) circle (2.5pt); %
  \filldraw (0.75,1) circle (2.5pt); %
  \filldraw (1.15,0.65) circle (2.5pt); %
  \filldraw (1.75,1) circle (2.5pt); %
  \filldraw (1.62,1.5) circle (2.5pt); %
  \filldraw (2.5,0.6) circle (2.5pt); %
  \filldraw (1.5,2) circle (2.5pt); %
        
  \draw[thick](0,0)--(3,0);%
  \draw[thick](1.5,2)--(0,0);%
  \draw[thick](1.5,2)--(1,0);%
  \draw[thick](1.5,2)--(2,0);%
  \draw[thick](1.75,1)--(0,0);%
    \node at (1.5,-0.75) {$N$};%
\end{tikzpicture}
\caption{The matroid $Q$ is the free $1$-cone of a parallel extension
  of $U_{2,4}$, while $N$ is not a free cone but has the same
  configuration.}
  \label{fig:needr>2}
\end{figure}
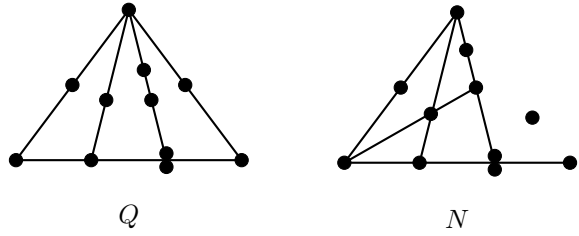

The example in Figure \ref{fig:needr>2} shows why we must assume that
$r(M)>2$ in the next result.

\begin{thm}\label{thm:coneconfigu}
  The free $m$-cone $Q$ of a matroid $M$ with no loops and with
  $r(M)>2$ is configuration unique.
\end{thm}

\begin{proof}
  We start from a key step that is established in the proof of
  \cite[Theorem 3.8]{freecone}: among all sets $\mathcal{C}$ of lines
  (i.e., rank-$2$ flats) in $\mathcal{Z}(Q)$ for which
  \begin{itemize}
  \item[(L1)] for each $L\in\mathcal{C}$, at most one proper, nonempty
    subset of $L$ is in $\mathcal{Z}(Q)$,
  \item[(L2)] $\bigvee\limits_{L\in \mathcal{C}}L=E(Q)$, and
  \item[(L3)] if $L, L'\in \mathcal{C}$ with $L\neq L'$, then
    $r(L\vee L')=3$,
  \end{itemize}
  there is a unique $\mathcal{C}$ for which $|\mathcal{C}|$ is
  maximal, and that $\mathcal{C}$, which we henceforth call
  $\mathcal{L}$, is the set of lines of $Q$ that contain $a$, so
  $|\mathcal{L}|$ is the number of rank-$1$ flats of $M$.  Note that
  whether a set $\mathcal{C}$ of lines in $\mathcal{Z}(Q)$ satisfies
  properties (L1)--(L3) can be deduced from the configuration alone,
  so $\mathcal{L}$ can be identified in the configuration.

  There is a maximal cyclic flat $F$ of $Q$ that contains no lines in
  $\mathcal{L}$; indeed, $F$ is the largest cyclic flat in
  $\mathcal{Z}(M)$.  This flat $F$ contains all rank-$1$ cyclic flats
  of $Q$, and for any $L\in\mathcal{L}$, we have
  $r_Q(F\join L)=r_Q(F)+1$ if and only if $F\cap L\ne\emptyset$.

  Assume that $N$ has the same configuration as $Q$. Thus, there is a
  lattice isomorphism $\Phi:\mathcal{Z}(Q)\to \mathcal{Z}(N)$ that
  preserves size and matroid rank.  As in several other arguments, it
  suffices to show that $\Phi$ is induced by a bijection
  $\phi:E(Q)\to E(N)$.

  We first show that one element of $E(N)$ is in all lines in
  $\Phi(\mathcal{L})$.  To see this, first note that the meet in
  $\mathcal{Z}(Q)$ of any two lines in $\mathcal{L}$ is $\emptyset$,
  so the same is true of the meet in $\mathcal{Z}(N)$ of any two lines
  in $\Phi(\mathcal{L})$.  Thus, the intersection of any two lines of
  $\Phi(\mathcal{L})$ is a singleton or empty.  Also, if all lines of
  $\Phi(\mathcal{L})$ were disjoint, then we would have
  $$|E(N)|\geq\sum_{L'\in\Phi(\mathcal{L})}|L'|
  =\sum_{L\in\mathcal{L}}|L|>|E(Q)|,$$ contrary to $N$ and $Q$ having
  the same configuration.  So assume that $L_1\cap L_2=\{a'\}$ for
  some $L_1,L_2\in\Phi(\mathcal{L})$.  Since $r(M)>2$, there are lines
  $L' \in\Phi(\mathcal{L})$ not in the plane $L_1\join L_2$.  For any
  such line $L'$, the planes $L'\join L_1$ and $L'\join L_2$ contain
  $a'$ and intersect in the line $L'$, so $a'\in L'$.  Applying the
  same argument using $L_1$ and $L'$ and any other line in
  $L_1\join L_2$ that is in $\Phi(\mathcal{L})$ shows that all such
  lines contain $a'$, so all lines in $\Phi(\mathcal{L})$ contain
  $a'$.

  Since all lines in $\Phi(\mathcal{L})$ contain $a'$ and $N$ has the
  same configuration as $Q$, counting shows that $E(N)$ is the
  disjoint union of the sets $\{a'\}$ and $L-\{a'\}$ as $L$ ranges
  over all lines in $\Phi(\mathcal{L})$.  Counting also shows that if
  $L\in \Phi(\mathcal{L})$ and $r_N(\Phi(F)\join L)=r_N(\Phi(F))+1$,
  then $L\cap \Phi(F)$ is a rank-$1$ flat of $N|\Phi(F)$ (which may or
  may not be cyclic).  With this, we can define $\phi:E(Q)\to E(N)$,
  namely,
  \begin{itemize}
  \item $\phi(a)=a'$;
  \item for each rank-$1$ flat $A$ of $M$ where $A$ does not consist
    of a coloop of $M$, let $L=\cl_Q(A\cup a)$, so $L\in\mathcal{L}$,
    and let $A'=\Phi(L)\cap \Phi(F)$; then let $\phi$ map $A$ onto
    $A'$ bijectively, and map $L-(a\cup F)$ onto
    $\Phi(L)-(a'\cup \Phi(F))$ bijectively;
  \item finally, for a line $L$ in $\mathcal{L}$ that contains no
    element of $F$, so $L\cap E(M)$ is a coloop of $M$, let $\phi$ map
    $L-a$ bijectively onto $\Phi(L)-a'$.
  \end{itemize}
  It follows that $\Phi(X)=\phi(X)$ for any cyclic flat of $Q$ that
  either contains $a$ or has rank $1$.  Also, $\Phi(F)=\phi(F)$.  Now
  consider any other cyclic flat $X$ of $M$ with $r_M(X)>1$.  The
  closure $\cl_Q(X\cup a)$ is a cyclic flat of $Q$ of rank $r_M(X)+1$,
  and $X=\cl_Q(X\cup a)\cap F$. Since $\Phi(\cl_Q(X\cup a))$ covers
  $\Phi(X)$ in both the lattice of flats of $N$ and the lattice of
  cyclic flats of $N$, we must have
  $\Phi(X)=\Phi(\cl_Q(X\cup a))\cap \Phi(F)$.  Thus,  as needed,
  $$\Phi(X)=\Phi(\cl_Q(X\cup a))\cap \Phi(F) =\phi(\cl_Q(X\cup a))\cap
  \phi(F)=\phi(X).\qedhere$$
\end{proof}

\begin{figure}
  \centering
\begin{tikzpicture}[scale=1]

  \filldraw (0,0) node[left] {\footnotesize $w$} circle (2.5pt); %
  \filldraw (2,0) node[right] {\footnotesize $x$} circle (2.5pt); %
  \filldraw (0.1,0.8) node[left] {\footnotesize $z$} circle (2.5pt); %
  \filldraw (1.9,0.8) node[right] {\footnotesize $y$} circle
  (2.5pt); %
  \filldraw (0.3,0.6) circle (2.5pt); %
  \filldraw (0.5,1) circle (2.5pt); %
  \filldraw (1.7,0.6) circle (2.5pt); %
  \filldraw (1.5,1) circle (2.5pt); %
  \filldraw (0.6,1.5) circle (2.5pt); %
  \filldraw (0.4,1.2) circle (2.5pt); %
  \filldraw (1.4,1.5) circle (2.5pt); %
  \filldraw (1.6,1.2) circle (2.5pt); %
   
  \draw[thick](0.1,0.8)--(0,0)--(2,0)--(1.9,0.8);%
  \draw[thick](1,2)--(0,0);%
  \draw[thick](1,2)--(2,0);%
  \draw[thick,dashed](0.1,0.8)--(1.9,0.8);%
  \draw[thick](0.1,0.8)--(1,2)--(1.9,0.8);%

  \node at (1,-0.75) {$Q_2(U_{3,4})\del a$};%
\end{tikzpicture}
\hspace{0.5cm}
  \begin{tikzpicture}[scale=1.3]
    \node[inner sep = 0.3mm] (em) at (2.2,-0.2) {\footnotesize
      $\emptyset$};%

    \node[inner sep = 0.3mm] (1a) at (0,0.5) {\footnotesize $W$};%
    \node[inner sep = 0.3mm] (1b) at (1,0.5) {\footnotesize $X$};%
        \node[inner sep = 0.3mm] (1c) at (2,0.5) {\footnotesize $Y$};%
    \node[inner sep = 0.3mm] (1d) at (3,0.5) {\footnotesize $Z$};%

    \node[inner sep = 0.3mm] (2a) at (-1,1.2)  {\footnotesize
      $W\cup X$};%
    \node[inner sep = 0.3mm] (2b) at (-0.1,1.2)  {\footnotesize
      $W\cup Y$};%
    \node[inner sep = 0.3mm] (2c) at (0.8,1.2)  {\footnotesize
      $W\cup Z$};%
        \node[inner sep = 0.3mm] (2d) at (1.7,1.2)  {\footnotesize
      $X\cup Y$};%
    \node[inner sep = 0.3mm] (2e) at (2.6,1.2)  {\footnotesize
      $X\cup Z$};%
    \node[inner sep = 0.3mm] (2f) at (3.5,1.2)  {\footnotesize
      $Y\cup Z$};%
        \node[inner sep = 0.3mm] (2g) at (4.5,1.2)  {\footnotesize
          $\{w,x,y,z\}$};%
        
    \node[inner sep = 0.3mm] (3) at (2.2,1.9)  {\footnotesize
      $W\cup X\cup Y\cup Z$};%

    \foreach \from/\to in {em/1a,em/1b,em/1c,em/1d,em/2g,
      1a/2a,1b/2a,1a/2b,1b/2d,1a/2c,1b/2e,1c/2b,1c/2d,1c/2f,
      1d/2c,1d/2e,1d/2f,
      2a/3,2b/3,2c/3,2d/3,2e/3,2f/3,2g/3} \draw(\from)--(\to);%

         \node at (2,-0.7) {$\mathcal{Z}(Q_2(U_{3,4}\del a)$};%
  \end{tikzpicture}
  \caption{The tipless $2$-cone of $U_{3,4}$ and its lattice of cyclic
    flats.  The set $W$ is the cyclic line that contains $w$, and
    likewise for $X$, $Y$, and $Z$.  }
  \label{fig:tipless}
\end{figure}
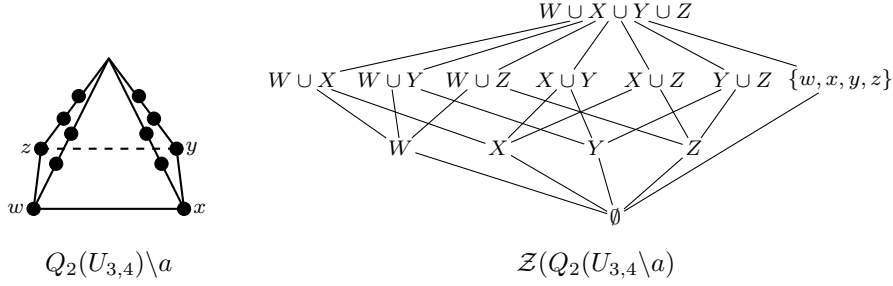

In \cite{freecone}, it is shown that if matroids $M$ and $N$ have the
same $\mathcal{G}$-invariant but are not isomorphic, then the free
$m$-cones $Q_m(M)$ and $Q_m(N)$ have the same $\mathcal{G}$-invariant
but different configurations.  Theorem \ref{thm:coneconfigu}
strengthens the conclusion: via free $m$-cones, we get non-isomorphic
configuration-unique matroids that have the same
$\mathcal{G}$-invariant.  Several variations on free $m$-cones are
also treated in \cite{freecone}.  The proof of Theorem
\ref{thm:coneconfigu} easily adapts to show that if $m>1$ and
$r(M)>2$, then the baseless free $m$-cone $Q_m(M)\del E(M)$ is
configuration unique.  However, the tipless free $m$-cone
$Q_m(M)\del a$ and the tipless/baseless free $m$-cone
$Q_m(M)\del (E(M)\cup a)$ need not be configuration unique.  We
illustrate this with $Q_2(U_{3,4})\del a$, which, along with its
lattice of cyclic flats, is shown in Figure \ref{fig:tipless}.  To get
a non-isomorphic matroid $N$ that has the same configuration as
$Q_2(U_{3,4})\del a$, consider the following sets and ranks:
\begin{itemize}
\item $\emptyset$ has rank $0$;
\item the sets $\{w,w',a\}$, $\{x,x',a\}$, $\{y,y',a\}$, and
  $\{z,z',a\}$ have rank $2$;
\item the sets $\{w,w',x,x',a,b\}$, $\{w,w',y,y',a,c\}$,
  $\{w,w',z,z',a,d\}$, $\{x,x',y,y',a,d\}$, $\{x,x',z,z',a,c\}$,
  $\{y,y',z,z',a,b\}$, and $\{w,x,y,z\}$ have rank $3$;
\item the set $\{w,w',x,x',y,y',z,z',a,b,c,d\}$ has rank $4$.
\end{itemize}
It is easy to check that properties (Z0)--(Z3) in Theorem
\ref{thm:axioms} hold, so this data defines a matroid $N$ which is
clearly not isomorphic to $Q_2(U_{3,4})\del a$ but has the same
configuration.

\section{
Which lattices come from non-configuration-unique matroids?
%
%  Tutte-equivalent transversal matroids with $\mathcal{Z}(M)\cong L$
}

As noted earlier, matroids for which the lattice of cyclic flats is a
chain (i.e., nested matroids) are Tutte unique \cite{AnnaThesis}.  In
this section, we prove the following theorem, which shows that any
lattice that is not a chain is isomorphic to the lattice of cyclic
flats of a matroid that is not even configuration unique.

\begin{thm}\label{thm:latticeresult}
  Let $L$ be a lattice that is not a chain.  There are pairs of
  non-isomorphic transversal matroids that have the same configuration
  and have their lattices of cyclic flats isomorphic to $L$.
\end{thm}

We will use one of the two constructions that we gave
in~\cite{cycflats} that, for a lattice $L$, produce transversal
matroids for which the lattices of cyclic flats are isomorphic to $L$.
We start by recalling the construction, which we illustrate in Figure
\ref{fig:latticeconstruction}.

Let $B=L-\{\hat{1}\}$ where $\hat{1}$ is the greatest element of $L$.
For each $z\in L$, let $V_z$ be the set
$\{y\in L \,:\, y\not \geq z\}$.  Observe that $V_x\subseteq V_z$ if
and only if $x\leq z$.  For each $z\in L$, let $S_z$ be a set of
$|V_z|+1$ elements, where $S_z\cap S_y=\emptyset$ whenever $z\ne y$,
and $S_z\cap B=\emptyset$.  Consider a $|B|$-vertex simplex $\Delta$.
Put one element of $B$ at each vertex of $\Delta$.  For each $z\in L$
put the points in $S_z$ freely in the face of $\Delta$ that is spanned
by $V_z$, and then delete $B$.  The resulting transversal matroid has
as cyclic flats the sets $F_z=\cup_{y\leq z} S_y$, for each $z\in L$;
also, $F_z\cap F_x=F_{z\meet x}$ for all $z,x\in L$, so the meet in
the lattice of cyclic flats is the intersection (as in the lattice of
all flats).  To give another view, the presentation consists of the
sets $A_y=\cup_{z\not\leq y}S_z$ for each $y\in B$.

\begin{figure}
  \centering
  \begin{tikzpicture}[scale=1]
    \node[inner sep = 0.3mm] (em) at (3.5,-0.5) {$\hat{0}$};%
    \node[inner sep = 0.3mm] (a) at (2.75,0.5) {$w$};%
    \node[inner sep = 0.3mm] (b) at (4.25,0.5) {$z$};%
    \node[inner sep = 0.3mm] (t) at (3.5,1.5) {$\hat{1}$};%

    \foreach \from/\to in {em/a,em/b,a/t,b/t} \draw(\from)--(\to);%

    \node at (3.5,-1.2) {$L$};%
        
    \draw[thick](5.5,-0.6)--(8.5,-0.6) --(7,1.7)--(5.5,-0.6);%
    \filldraw (5.95,0.1) node {} circle (2.5pt); %
    \filldraw (8.05,0.1) node {} circle (2.5pt);%
    \filldraw (6.3,0.6) node {} circle (2.5pt);%
    \filldraw (7.7,0.6) node {} circle (2.5pt);%
    \filldraw (6.6,1.1) node {} circle (2.5pt);%
    \filldraw (7.4,1.1) node {} circle (2.5pt);%
    \filldraw (7.1,0.3) node {} circle (2.5pt);%
    \filldraw (6.9,-0.3) node {} circle (2.5pt); %
    \filldraw (7.5,0) node {} circle (2.5pt); %
    \filldraw (6.5,0) node {} circle (2.5pt); %

    \node at (5.35,-0.8) {{\color{black!60}$w$}};%
    \node at (8.65,-0.8) {{\color{black!60}$z$}};%
    \node at (7,1.95) {{\color{black!60}$\hat{0}$}};%
    \node at (6.3,1.1) {$e$};%
    \node at (7.7,1.1) {$f$};%
    \node at (7,-1.2) {$M$};%

    \draw[thick](10,-0.6)--(13,-0.6) --(11.5,1.7)--(10,-0.6);%
    \filldraw (10.45,0.1) node {} circle (2.5pt); %
    \filldraw (12.55,0.1) node {} circle (2.5pt);%
    \filldraw (10.8,0.6) node {} circle (2.5pt);%
    \filldraw (12.2,0.6) node {} circle (2.5pt);%
    \filldraw (11.3,0.9) node {} circle (2.5pt);%
    \filldraw (11.5,1.7) node {} circle (2.5pt);%
    \filldraw (11.6,0.3) node {} circle (2.5pt);%
    \filldraw (11.4,-0.3) node {} circle (2.5pt); %
    \filldraw (12,0) node {} circle (2.5pt); %
    \filldraw (11,0) node {} circle (2.5pt); %

    \node at (11.5,1.95) {$e$};%
    \node at (11.5,0.9) {$f$};%

    \node at (11.5,-1.2) {$M'$};%    
  \end{tikzpicture}
  \caption{In this example of the construction used to prove Theorem
    \ref{thm:latticeresult}, we have $V_{\hat{0}}=\emptyset$,
    $V_w=\{\hat{0},z\}$, $V_z=\{\hat{0},w\}$, and
    $V_{\hat{1}}=\{\hat{0},w,z\}$.  The three points in $S_w$ are on
    the edge of the simplex labeled by the elements of $V_w$, namely,
    $\hat{0}$ and $z$.  We have omitted the loop that $S_\emptyset$
    contributes.  We get $M'$ by having $e$ bump $f$ in $M$.}
  \label{fig:latticeconstruction}
\end{figure}
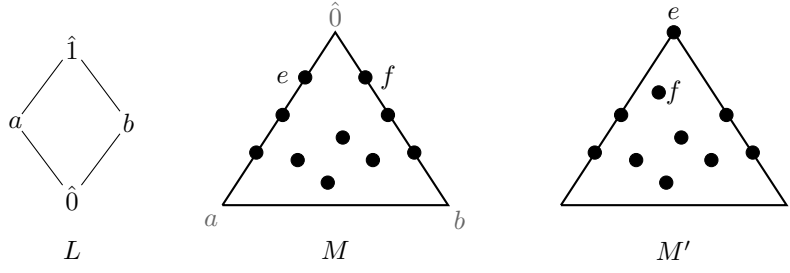

\begin{proof}[Proof of Theorem \ref{thm:latticeresult}]
  Observe that the greatest cyclic flat of a matroid $M$ without
  coloops covers exactly one cyclic flat, say $X$, in $\mathcal{Z}(M)$
  if and only if we obtain $M$ from $M|X$ by adding $r(M)-r(X)$
  elements of $E(M)-X$ as coloops and then adding the remaining
  elements of $E(M)-X$ by free extension.  It follows that if we prove
  the result for lattices where $\hat{1}$ covers at least two
  elements, then it holds for all lattices that are not chains.  So we
  assume that the greatest element $\hat{1}$ of $L$ covers at least
  two elements of $L$, say $z$ and $w$.
  
  Let $M$ be the matroid constructed in the paragraph before the
  proof.  Let $r=|B|$.  Thus, $|V_z|=|V_w|=r-1$, so
  $r_M(F_z)= r_M(F_w)=r-1$. Now $r_M(F_{z\meet w}) \leq r-3$ since
  $z$, $w$, and $z\meet w$ are not in $V_{z\meet w}$.  Since
  $F_z\cap F_w=F_{z\meet w}$, the pair $(F_z, F_w)$ is not modular.
  For any $e\in S_z$ and $f\in S_w$, we have, in the notation of
  Theorem \ref{thm:twofilters}, $\mathcal{Z}_e=\{F_z,E(M)\}$ and
  $\mathcal{Z}_f=\{F_w,E(M)\}$, so that theorem applies; thus, $e$ can
  bump $f$ in $M$ to obtain a matroid $M'$ that is not isomorphic to
  $M$ but has the same configuration as $M$.

  We show that $M'$ is transversal by applying Theorem \ref{thm:mi}
  and the remarks that follow it.  Let $\mathcal{F}'$ be an antichain
  of cyclic flats of $M'$ with $|\mathcal{F}'|\geq 3$.  Since
  $|\mathcal{F}'|\geq 3$, the antichain $\mathcal{F}'$ contains a
  cyclic flat of $M'$ other than $E(M')$, $F_z$, and $(F_w-f)\cup e$.
  Thus, since $\mathcal{Z}_e=\{F_z,E(M)\}$ and
  $\mathcal{Z}_f=\{F_w,E(M)\}$, we have $e,f\not\in \cap\mathcal{F}'$.
  Let $\mathcal{F}$ be the corresponding antichain in $\mathcal{Z}(M)$
  (so, replace $(F_w-f)\cup e$ by $F_w$ if it is in $\mathcal{F}'$;
  otherwise $\mathcal{F}=\mathcal{F}'$).  Since
  $e,f\not\in \cap\mathcal{F}'$, we have
  $\cap\mathcal{F}'=\cap\mathcal{F}$. Also,
  $r_{M'}(\cap\mathcal{F}')=r_M(\cap\mathcal{F})$ by Equation
  (\ref{eq:cftorank}) since, with $e,f\not \in \cap\mathcal{F}'$, the
  cyclic flats $F=(F_w-f)\cup e$ and $F=F_w$ in which $\mathcal{Z}(M)$
  and $\mathcal{Z}(M')$ differ yield the same sum
  $r(F)+|(\cap\mathcal{F}')-F|$.  Also, the right side of Inequality
  (\ref{eq:mi}) depends only on the configuration, which $M$ and $M'$
  share.  Thus, Inequality (\ref{eq:mi}) holds for $\mathcal{F}'$ in
  $M'$ since it holds for $\mathcal{F}$ in $M$, so $M'$ is
  transversal.
\end{proof}

To close, we note a stronger version of this result.  First, delete
the loop, the element of $S_{\hat{0}}$, from the matroids $M$ and $M'$
constructed above.  The elements of $S_{\hat{1}}$ are in no proper
cyclic flats of either $M$ or $M'$, so any circuit of $M$ that
contains an element of $S_{\hat{1}}$ spans $M$, and likewise for $M'$.
Thus, $M$ and $M'$ have spanning circuits but no loops and so are
connected.  As shown in \cite{texpansion}, applying the operation of
$t$-expansion to $M$ and $M'$ yields matroids $N$ and $N'$ that have
the same configuration, are transversal, and, by letting $t$ be large
enough, have arbitrarily high connectivity.  (The operation of
$t$-expansion, introduced in \cite{texpansion}, can be seen as
magnifying all sets and ranks in the lattice of cyclic flats by a
factor of $t$ \cite[Definition 3.1]{texpansion} or as the result of
taking a matroid and replacing each element by a set of $t$ parallel
elements (or $t$ loops if the element is a loop), and taking the
matroid union of $t$ copies of this matroid \cite[Theorem
3.11]{texpansion}.  The fact that $t$-expansions inherit the
properties of interest follows from \cite[Lemmas 3.2 and 3.3, and
Theorem 3.13]{texpansion}.)  Similar remarks apply trivially to
vertical connectivity and branch-width since the vertical connectivity
of $M$ and $M'$ is their rank, and their branch-width is one more than
their rank.

\section*{Acknowledgements}

The second author was supported by the Grant PID2023-147202NB-I00
funded by MICIU/AEI/10.13039/501100011033.

\end{document}